\documentclass[final, 11pt]{article}
\usepackage{tikz}
\usepackage{upgreek,ctable}
\usepackage{bm} 
\usepackage{amsmath}
\usepackage{amsthm}
\usepackage{xcolor,colortbl}
\usepackage{amsfonts}
\usepackage{makeidx}
\usepackage{amsmath}
\usepackage{accents}
\usepackage{hyperref}
\usepackage{nameref}
\usepackage{amssymb}
\usepackage{bm}
\usepackage{enumerate}
\usepackage{cite}
\usepackage{showkeys}
\usepackage{subfig}
\usepackage[english]{babel}
\usepackage{psfrag}
\usepackage{subfig}
\usepackage{color}
\usepackage{float}
\usepackage{amscd}
\usepackage[mathscr]{euscript}
\usepackage{lipsum}
\usepackage{epstopdf}
\usepackage{algorithm}
\usepackage{algpseudocode,caption}
\usepackage{graphicx}
\everymath{\displaystyle}
\newcommand{\N}{\ensuremath{\mathbb{N}}}
\newcommand{\G}{\ensuremath{\mathbb{G}}}

\newcommand{\R}{\ensuremath{\mathbb{R}}}
\newcommand{\C}{\ensuremath{\mathbb{C}}}

\renewcommand{\d}{\mathrm{d}}

\renewcommand{\O}{\mathcal{O}}

\newcommand{\p}{\mathbf{p}}
\newcommand{\q}{\mathbf{q}}

\renewcommand{\Re}{\mathrm{Re}\,}
\renewcommand{\Im}{\mathrm{Im}\,}

\newcommand{\E}{{\mathbf{E}}}

\newcommand{\F}{{\bf F}}

\usepackage{mathtools}

\renewcommand{\div}{\mathrm{div}}
\newcommand{\curl}{\mathrm{curl}\,}

\newcommand{\x}{\mathbf{x}}
\newcommand{\y}{\mathbf{y}}
\newcommand{\z}{\mathbf{z}}

\newcommand{\bw}{\mathbf{w}}

\everymath{\displaystyle}
\newtheorem{defi}{Definition}
\newtheorem{lemma}[defi]{Lemma}
\newtheorem{theorem}[defi]{Theorem}

\newtheorem{remark}[defi]{Remark}
\begin{document}
\title{A direct reconstruction method  for radiating  sources in Maxwell’s equations with single-frequency data}

\author{Isaac Harris\thanks{ Department of Mathematics, Purdue University, West Lafayette, IN 47907; (\texttt{harri814@purdue.edu})} \and Thu Le\thanks{Department of Mathematics, University of Wisconsin-Madison, Madison, WI 53706; (\texttt{tle38@wisc.edu})} \and Dinh-Liem Nguyen\thanks{Department of Mathematics, Kansas State University, Manhattan, KS 66506; (\texttt{dlnguyen@ksu.edu})}   
}
\date{}
\maketitle

\begin{abstract}
This paper presents a fast and robust numerical method for reconstructing point-like sources in the time-harmonic Maxwell's equations given Cauchy data at a fixed frequency. This is an electromagnetic inverse source problem with broad applications, such as antenna synthesis and design, medical imaging, and pollution source tracing.  We introduce new imaging functions and a computational algorithm to determine the number of point sources, their locations, and associated moment vectors, even when these vectors have notably different magnitudes. The number of sources and locations are estimated using significant peaks of the imaging functions, and the moment vectors are computed via explicitly simple formulas.  The theoretical analysis and stability of the imaging functions are investigated, where the main challenge lies in analyzing the behavior of the dot products between the columns of the imaginary part of the Green's tensor and the unknown moment vectors. Additionally, we extend our method to reconstruct small-volume sources using an asymptotic expansion of their radiated electric field.   
We provide numerical examples in three dimensions to demonstrate the performance of our method.

\end{abstract}
\sloppy


\bigskip
\section{Introduction}
Inverse source problems have a wide range of applications in scientific areas. In antenna synthesis and design, the problem involves determining the current distribution required to produce a desired radiation pattern \cite{Leone2018}. In medical imaging, such as electroencephalography (EEG) and magnetoencephalography (MEG), they help to reconstruct the electrical activity of the brain based on boundary measurements, thus enabling the diagnosis of internal abnormalities \cite{Hamalainen1993}. 
Inverse source problems are also used to trace pollution sources in the environment  \cite{Badia2002}. These problems have garnered significant attention from researchers over the past thirty years. For extensive results on the inverse source problems for the Helmholtz equation, see \cite{DZhang2019, Kress2013, Nara2008, Loc2023, Loc2019, Badia2019, Kui2019}. We are interested in the electromagnetic inverse source problem for the time-harmonic Maxwell’s equations at a fixed frequency. We aim to provide an effective numerical method for reconstructing the number of unknown point-like sources, along with their locations and possibly their moment vectors, from measured boundary Cauchy data. 

The presence of non-radiating sources introduces difficulty to time-harmonic inverse source problems for wave equations, particularly Maxwell's equations, potentially leading to nonunique solutions \cite{Dassios2005, Bleistein77, Hauer2005, Monk2006}. To address this challenge, an approach is to utilize multi-frequency boundary measurements. We direct the reader to recent studies that provide significant mathematical insights into the uniqueness, stability, and numerical findings of inverse source problems using multi-frequency data  \cite{Bao2020, Cheng2016, Wang2018, Valdivia2012, Bao2010, Isakov2021, Wang2019, Wang2022}. For the case of a single frequency data, it is necessary to impose additional constraints on the sources to obtain a unique solution to the inverse source problem. It was proved in \cite{Monk2006} that Cauchy data can uniquely determine multiple point sources or surface currents at a fixed frequency. 

Several numerical methods have been developed for reconstructing multiple electromagnetic point sources with single-frequency data. From a boundary integral-based relation between the Cauchy data and the unknown source term, the authors in \cite{Inui2005} developed a strategy to choose a certain ``vector-valued weighting function" in the boundary integral formula to develop an identification method for the number of sources, locations, and moment vectors. 
Nevertheless, the method poses a challenge in choosing the weighting function in its implementation. Moreover, based on a similar boundary integral-based relation, the authors in \cite{Badia2013, Abdelaziz2017} proposed a different way to choose the weighting function and studied an algebraic method to identify sources. This algebraic method relies on the rank of a Hankel matrix to determine the number of unknown sources.
However, if the rank is not low, it is challenging to determine this rank accurately because the values of the last singular values are small and quite close to each other.  We also refer to \cite{Griesmaier2018} for a two-step numerical method using a windowed Fourier transform to determine the support of small sources from far-field data, though this method does not address the reconstruction of the moment vectors of these sources.

In this paper, we study a numerical method that also relies on a boundary integral-based formula. By choosing a simple weighting function, we develop novel imaging functions and a numerical algorithm that allow us to determine the unknown sources without having to deal with the challenges faced by the methods in \cite{Inui2005, Badia2013, Abdelaziz2017}. Furthermore, our approach is easy to implement and avoids expensive iterative computations, resulting in a more efficient performance when working with a larger number of sources. It also demonstrates robustness against noisy data without the need for additional regularization techniques.
 {The method presented here is a non-iterative or direct method. Similar methods have been used to reconstruct small volume and extended scatterers via the MUSIC Algorithm and Factorization method \cite{Ammari2001,Ammari2007,Kirsch2004}. These methods need multi-static data as the scattered field is induced by an incident field.}
This work is motivated by our results in \cite{Harris2023} for the Helmholtz equation. 
However, the analysis of the new imaging functions is more challenging and requires an innovative approach due to the complicated behavior of the imaginary part of the  Green's tensor, compared to the free-space scalar Green's function discussed in \cite{Harris2023}. 
We analyze key properties of the Green's tensor and establish a theoretical justification for our imaging functions. Our functions are also flexible as they enable imaging at arbitrary distances from the data boundary. Detecting sources with significantly different magnitudes of moment vectors is another challenge, as weaker sources are dominated by stronger ones, making it very difficult to detect them. We introduce a fast algorithm for identifying sources with (possibly complex) moment vectors of varying magnitudes. Additionally, we define an alternative function for source reconstruction and compare it with the proposed one.

The remainder of this paper is organized as follows. Section \ref{se:pointsource} is devoted to the identification of point sources. Section \ref{se:smallballs} concerns detecting point-like sources with small-volume support. Numerical examples in 3D are presented in Section \ref{se:results} to validate our method. Finally, a conclusion is given in Section \ref{se:conclu}.

\section{Identification of electromagnetic point sources}
\label{se:pointsource}
This section focuses on the inverse problem for point sources governed by time-harmonic Maxwell's equations. We consider a set of $N\in\N$ electromagnetic point sources, each characterized by its location $\x_j \in \mathbb{R}^3$ and a nonzero moment vector $\p_j \in \mathbb{C}^3$ for $j=1,2, \ldots,N$. These point sources are represented by the Delta distributions $\delta_{\x_j}$. Let $k>0$ be the wavenumber. Suppose that the sources are well-separated, 
\begin{align}
    \label{assump:well-sep}
    \text{dist}_{i\ne j}(\x_i,\x_j)\gg \lambda,
\end{align}
where $\lambda=2\pi/k$ is the wavelength. \\ 
\newpage 
\noindent We assume that these point sources generate the radiated field $\E: \R^3\rightarrow \C^3$ satisfying the following model
\begin{align}
&\curl \curl \E - k^2\E =  -\sum_{j=1}^N \delta _{\x_j}\p_j, \quad \text{in } \R^3,\label{eq:maxwell}
\\
&\lim_{|\x|\rightarrow \infty}|\x|\left(\curl \E \times \frac{\x}{|\x|} - ik\E\right)=0,
    \label{eq:silver}
\end{align}
where the  Silver-Müller radiation condition  
\eqref{eq:silver}  is assumed to hold uniformly for all directions $\x/|\x|$.  {The forward problem of finding $\E$ for known sources is well-posed (see \cite{Monk2003}). We proceed under this assumption to study the inverse problem.}

Let us introduce the Green's tensor $\G(\x,\y)$ for the problem \eqref{eq:maxwell}-\eqref{eq:silver}
\begin{equation}
    \G(\x,\y) = \Phi(\x,\y)I_3+\frac{1}{k^2}\nabla_{\x} \div_{\x} (\Phi(\x,\y)I_3),
    \label{greentensor}
\end{equation}
where 
\begin{equation*} 
\label{green}
 \Phi(\x,\y)= 
\frac{e^{ik|\x-\y|}}{4\pi|\x-\y|},
\end{equation*}
for any $\x\ne \y$ in $\R^3$. Here, $I_3$ is the $3\times 3$ identity matrix. The operators $\nabla$ and $\div$ are applied to the tensor as column-wise. 

 It is well-known that the unique electric field $\E$ can be represented as
 
\begin{equation}
\E(\x) = \sum_{j=1}^N \G(\x,\x_j)\p_j, \quad \text{for }\x \in \R^3, \, \x \neq \x_j.
\label{reprep}
\end{equation}
Let $\Omega \subset \R^3$ be a  bounded and open domain with Lipschitz boundary ${\partial \Omega}$. Assume that source locations $\x_j\in \Omega$ for $j=1,2, \ldots,N$ and  that the radiated field $\E$ is measured on ${\partial \Omega}$. Denote by ${\bm{\nu}}$  the unit outward normal vector to ${\partial \Omega}$.  We are interested in solving the following inverse problem.
\vspace{0.3cm}

\textbf{Inverse problem.} Given the Cauchy data  $\mathbf{E}$ and $\curl \mathbf{E} \times \boldsymbol{\nu}$ measured on  $\partial \Omega$ at a fixed wavenumber $k > 0$,  determine the total number of unknown point sources $N$, along with their locations $\mathbf{x}_j$ and moment vectors $\mathbf{p}_j$ for $j = 1,2, \ldots,N$.
\vspace{0.3cm}

We now study a function that will serve as a base for our imaging functions to reconstruct unknown sources. This function is motivated by the result in~\cite{Harris2023}. 
Let $\z \in \R^3$ be sampling points and $\q\in \R^3$ be a fixed nonzero vector, we define
\begin{equation}
\label{eq:im}
    I(\z,\q) :=  \int_{\partial \Omega} \curl (\Im{\G}(\x,\z)\q)\times \bm{\nu} \cdot \E(\x)-\curl \E(\x) \times \bm{\nu}\cdot\Im{\G}(\x,\z)\q\,\d s(\x).
\end{equation}
\\
The following lemma is very useful for the analysis of our numerical method.
\begin{lemma} For any sampling point $\z\in \R^3$, the base function  $I(\z,\q)$ satisfies 
    \begin{align*}
         I(\z,\q)  = \sum_{j=1}^N \p_j\cdot \Im\G(\x_j,\z)\q.
    \end{align*}
    \label{theo1}
\end{lemma}
\begin{proof}
    Let $\y\in \R^3$ be a point in $\Omega$. Then, for any vectors $\p\in \C^3$ and $\q\in\R^3$, using integration by parts and the Divergence Theorem, we have
    \begin{align*}
        \int_{\partial \Omega} \curl (\Im{\G}(\x,\z)\q)\times \bm{\nu} \cdot \G(\x,\y)\p\text{ }\d s(\x) =  \int_\Omega &-\curl \curl (\Im{\G}(\x,\z)\q) \cdot \G(\x,\y)\p
        \\ &+ \curl (\Im{\G}(\x,\z)\q) \cdot \curl ({\G}(\x,\y)\p)\,\d\x 
    \end{align*}
    and 
     \begin{align*}
        -\int_{\partial \Omega} \curl (\G(\x,\y)\p)\times \bm{\nu}\cdot\Im{\G}(\x,\z)\q\text{ } \d s(\x) =  
        \int_\Omega &\curl \curl ({\G}(\x,\y)\p) \cdot \Im\G(\x,\z)\q
        \\ &- \curl ({\G}(\x,\y)\p) \cdot \curl (\Im{\G}(\x,\z)\q) \,\d\x .
    \end{align*}
    Adding these expressions side by side and using the following Green's tensor identities
    \begin{align*}
   & \curl_{\x} \curl_{\x} ({\G}(\x,\y)\p) - k^2{\G}(\x,\y)\p= \delta (\x-\y) \p,\\
   &       \curl_{\x} \curl_{\x} (\Im{\G}(\x,\z)\q) - k^2\Im\G(\x,\z)\q  = 0,
    \end{align*}
    we obtain that  
\begin{align}
&\int_{\partial \Omega} \curl (\Im{\G}(\x,\z)\q)\times \bm{\nu} \cdot \G(\x,\y)\p-\curl (\G(\x,\y)\p)\times \bm{\nu}\cdot\Im{\G}(\x,\z)\q \,\d s(\x) \notag \\
        &= \int_\Omega \curl \curl ({\G}(\x,\y)\p) \cdot \Im\G(\x,\z)\q -  \G(\x,\y)\p \cdot \curl \curl (\Im{\G}(\x,\z)\q) \,\d\x\notag 
        \\
       & = \int_\Omega   \left(\curl \curl ({\G}(\x,\y)\p) - k^2\G(\x,\y)\p
        \right) \cdot \Im\G(\x,\z)\q \,\d\x \notag \\
        &= \p \cdot \Im\G(\y,\z)\q. \label{eq:id}
    \end{align}
Now substituting $\y=\x_j$ and $\p=\p_j$ into  \eqref{eq:id} for $j=1,2, \ldots,N$, and 
by the representation of the radiated field $\E(\x)$ in (\ref{reprep}), we derive 
\begin{align*}
    I(\z,\q) &=\int_{\partial \Omega} \curl (\Im{\G}(\x,\z)\q)\times \bm{\nu} \cdot \sum_{j=1}^N\G(\x,\x_j)\p_j-\curl \left(\sum_{j=1}^N\G(\x,\x_j)\p_j\right) \times \bm{\nu} \cdot \Im{\G}(\x,\z)\q \,\d s(\x)
    \\
    &= \sum_{j=1}^N\int_{\partial \Omega} \curl (\Im{\G}(\x,\z)\q)\times \bm{\nu}\cdot \G(\x,\x_j)\p_j-\curl (\G(\x,\x_j)\p_j) \times \bm{\nu}\cdot\Im{\G}(\x,\z)\q\, \d s(\x)
    \\&=\sum_{j=1}^N \p_j\cdot \Im\G(\x_j,\z)\q.
\end{align*}    
\end{proof}

The next lemma is important to the analysis of the imaging functions.
\begin{lemma} 
\label{lem:imgq} For any $\x=(x_1,x_2,x_3)^\top ,\y=(y_1,y_2,y_3)^\top$ and $\q=(q_1,q_2,q_3)^\top$ in $\R^3$, we have
    \begin{align}
    \Im\G(\x,\y)\q 
    = kj_0(k|\bw|)\frac{\q|\bw|^2-(\q\cdot \bw)\bw}{4\pi|\bw|^2}+\left(j_0(k|\bw|)-\cos(k|\bw|)\right)\frac{3(\q\cdot \bw)\bw-\q|\bw|^2}{4\pi k|\bw|^4} \label{imgqq}, 
\end{align}
where $\bw=\x-\y=(w_1,w_2,w_3)^\top$  {and $j_0(x) = \sin(x)/x$.  }
Furthermore,
$$\lim_{|\bw|\rightarrow 0} \Im\G(\x,\y)\q= \frac{k}{6\pi}\q:=\Im\G(\x,\x)\q.$$
{Consequently, the matrix $\Im\G(\x,\x)$ is invertible, and $$[\Im\G(\x,\x)]^{-1}=\frac{6\pi}{k}I_3.$$}
\end{lemma}
\begin{proof}
    { We know that for $\bw\ne 0$, taking the imaginary part of the Green's tensor yields 
    \begin{align*}
         \Im\G(\x,\y) = \Im\Phi(\x,\y)I_3+\frac{1}{k^2}\Im \left(\nabla_{\x} \div_{\x} (\Phi(\x,\y)I_3)\right),
    \end{align*}
    where $$ \Phi(\x,\y) =   
        \frac{e^{ik|\bw|}}{4\pi|\bw|},\quad \quad \Im \Phi(\x,\y) =   
        \frac{k}{4\pi}j_0(k|\bw|),
$$ 
A direct calculation gives 
the formula in \eqref{imgqq}. We rewrite it as
\begin{align}
    \label{eq:pimgqq}
     \Im\G(\x,\y)\q  &= \frac{k}{6\pi} \q j_0 (k|\bw|)-\frac{k}{12\pi}\q \left( \frac{3\left(j_0(k|\bw|)-\cos(k|\bw|)\right)}{k^2|\bw|^2} - j_0(k|\bw|)
    \right) \notag
    \\ &\hspace{1.8 cm}+\frac{k \bw(\q\cdot \bw)}{4\pi |\bw|^2}\left( \frac{3\left(j_0(k|\bw|)-\cos(k|\bw|)\right)}{k^2|\bw|^2} - j_0(k|\bw|)
    \right).
\end{align}
Now, using Taylor expansion around $\bw=0$ gives
\begin{align*}
&j_0(k|\bw|)
=1-\frac{k^2|\bw|^2}{6}+\frac{k^4|\bw|^4}{120}+\mathcal{O}(|\bw|^6),\\
   & \frac{3\left(j_0(k|\bw|)-\cos(k|\bw|)\right)}{k^2|\bw|^2} - j_0(k|\bw|)
=\frac{k^2|\bw|^2}{15}-\frac{k^4|\bw|^4}{210}+\mathcal{O}(|\bw|^6).
\end{align*}
Substituting these expansions into \eqref{eq:pimgqq},  we then get 
    \begin{align}
 \Im\G(\x,\y)\q &= \frac{k}{6\pi} \q \left(1-\frac{k^2|\bw|^2}{6}+\mathcal{O}(|\bw|^4)\right)-\frac{k}{12\pi} \q \left( \frac{k^2|\bw|^2}{15}+\mathcal{O}(|\bw|^4)
    \right)\notag \\
      &\hspace{5.1cm}+\frac{3k^3\bw(\q\cdot \bw)}{12\pi }\left(  \frac{1}{15}+\mathcal{O}(|\bw|^2)
    \right) \notag \\
    &=\frac{k}{6\pi}  \q-\frac{k^3}{60\pi}\left(2 \q|\bw|^2-\bw(\q\cdot \bw)\right)  + \mathcal{O}(|\bw|^4). \label{taylor}
\end{align}
Thus, $$\lim_{|\bw|\rightarrow 0} \Im\G(\x,\y)\q= \frac{k}{6\pi}\q:= \Im\G(\x,\x)\q .$$ Choosing $\q={\bf e}_i$ where $i=1,2,3$, $ {\bf{e}}_1 := (1,0,0)^\top,{\bf{e}}_2 := (0,1,0)^\top$ and ${\bf{e}}_3 := (0,0,1)^\top$ gives us columns of matrix $\Im\G(\x,\x)$. That implies $$\Im\G(\x,\x)= \frac{k}{6\pi}I_3,$$ and its inverse matrix $$[\Im\G(\x,\x)]^{-1}= \frac{6\pi}{k}I_3.$$
}
\end{proof}
\begin{remark}  We can see from \eqref{imgqq} that  for any nonzero vector $\p\in \C^3$, the function $| \p \cdot \Im\G(\x,\y)\q|$ decays as $\y$ moves away from $\x$, with $| \p \cdot \Im\G(\x,\y)\q |= \O\left(\text{dist}(\x,\y)^{-1}\right)$.
By substituting  $\p=\p_j$, $\x=\x_j$ and $\y=\z$, for $j=1,2,\ldots, N$, into this dot product, and applying Lemma \ref{theo1}, alongside the triangular inequality, we obtain that
$$|I(\z,\q)| = \mathcal{O} \Big ( \text{dist}(\z,{\bf X})^{-1}\Big ), \text{ as } \text{dist}(\z,{\bf X}) \rightarrow \infty,$$ 
 where  ${\bf X} := \{\x_j : j=1, 2, \ldots, N\}$ represents the set of source locations.
 \label{re:Idecay}
 \end{remark}

In general, the function $I(\z,\q)$ does not always have peaks at source locations due to the complex behavior of  $\p \cdot \Im \G(\x,\z)\q$. For example, Figure \ref{fig:posfail}(b) shows that 
$|I(\z,\q)|$ 
with $\q=(1,1,1)^\top$ 
fails to correctly locate a point source. 
To improve this, for a fixed positive integer $s$, we propose a novel imaging function
\begin{align}
\widetilde I_s(\z) := |I(\z,{\bf{e}}_1)|^s +  |I(\z,{\bf{e}}_2)|^s +  |I(\z,{\bf{e}}_3)|^s.
\label{tildeI}
\end{align}
We later see (Theorem \ref{globaldominates}) that $s$ plays a crucial role in the resolution analysis of our imaging functions.

The function $I(\z, {\bf{e}}_i)$ is defined in \eqref{eq:im} with the choice $\q = {\bf{e}}_i$ for $i=1,2,3$. See Figure \ref{fig:posfail}(c) for an illustrative example.
\vspace{0.05cm}
\\
    To  investigate the behavior of $\widetilde I_s(\z)$, we refer to the expression obtained from Lemma \ref{theo1}, 
    \begin{align}
    \label{imzei}
        I(\z,{\bf{e}}_i)=\sum_{j=1}^N \p_j\cdot \Im \G(\x_{j},\z){\bf{e}}_i,\quad  \text{ for  }i=1,2,3.
    \end{align}
Hence, it is essential to study properties of  $\p_{} \cdot \Im \G(\x_{},\z){\bf{e}}_i$ for $i=1,2,3$. We begin by considering the case of real moment vectors. We later extend the analysis to study the case of complex moment vectors. We prove in Lemma \ref{atleat1peak} that for any nonzero vector $\p \in \R^3$, at least one of the terms 
$$|\p_{} \cdot \Im \G(\x_{},\z){\bf{e}}_1|^s,\quad |\p_{} \cdot \Im \G(\x_{},\z){\bf{e}}_2|^s,\quad \text{or} \quad |\p_{} \cdot \Im \G(\x_{},\z){\bf{e}}_3|^s$$ 
achieves its maximum at $\z=\x$ in a small neighborhood of $\x$. Ideally, if all three terms are maximal at $\z=\x$, their sum also attains maximum there in this neighborhood for any integer $s>0$, see Figure \ref{fig:Ie_all_peak}. Conversely, if one or two of the three terms reach a maximum at 
$\z=\x$ in a small neighborhood of $\x$, for some integer $s>0$, these maxima significantly dominate the other term(s) that do not peak at $\x$ in this neighborhood. Refer to Figure \ref{fig:Ie_one_peak} and Theorem \ref{globaldominates} for more details. As a result, for any nonzero vector $\p\in\R^3$ and  some integer $s>0$, 
\begin{align}
    |\p_{} \cdot \Im \G(\x_{},\z){\bf{e}}_1|^s+|\p_{} \cdot \Im \G(\x_{},\z){\bf{e}}_2|^s+  |\p_{} \cdot \Im \G(\x_{},\z){\bf{e}}_3|^s
   \notag
\end{align}
attains a maximum at $\z=\x$ in a small neighborhood of $\x$. This motivates the formulation of $\widetilde {I_s}(\z)$ in (\ref{tildeI}). 
\begin{figure}[ht]
    \centering
    \subfloat[ $|I(\z,\q)|$ with $\p=(-17,-7,-8)^\top$ ]
    {\begin{tikzpicture}
    \node[anchor=south west,inner sep=0] (image) at (0,0) {\includegraphics[width=5.1cm]{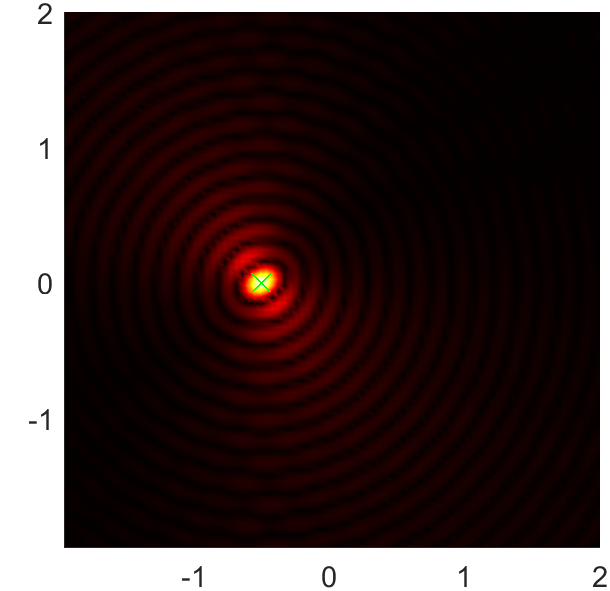}};
    \begin{scope}[x={(image.south east)},y={(image.north west)}]
      \node[anchor=north] at (.53,0) {\small{$x$}};
      \hspace{0.2cm}
      \node[anchor=south,rotate=90] at (0,.52) {\small{$y$}};
    \end{scope}
  \end{tikzpicture}  }
 \hspace{0.4cm}
\subfloat[$|I(\z,\q)|$ with $\p=(17,-7,-8)^\top$]
    {\begin{tikzpicture}
    \node[anchor=south west,inner sep=0] (image) at (0,0) {\includegraphics[width=5.1cm]{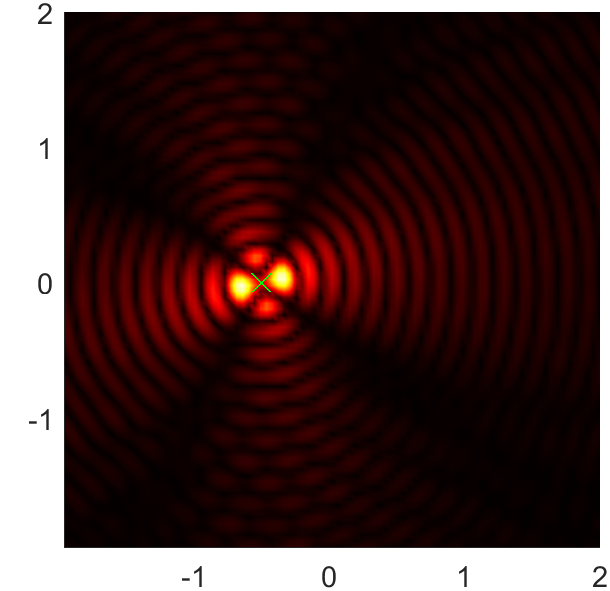}};
    \begin{scope}[x={(image.south east)},y={(image.north west)}]
      \node[anchor=north] at (.53,0) {\phantom{\scriptsize{$x$}}};
    \end{scope}
  \end{tikzpicture}  }
   \hspace{0.4cm}
\subfloat[$\widetilde I_{1}(\z)$ with $\p=(17,-7,-8)^\top$]
    {\begin{tikzpicture}
    \node[anchor=south west,inner sep=0] (image) at (0,0) {\includegraphics[width=5.1cm]{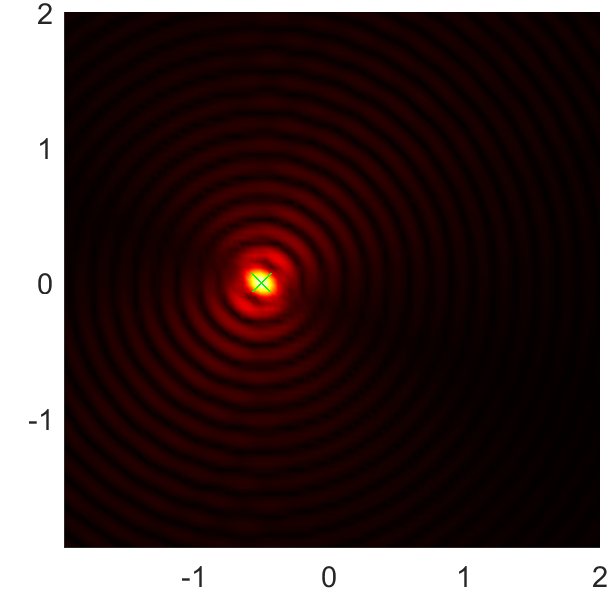}};
    \begin{scope}[x={(image.south east)},y={(image.north west)}]
      \node[anchor=north] at (.53,0) {\phantom{\scriptsize{$x$}}};
    \end{scope}
  \end{tikzpicture}  }
    \caption{Reconstruction results for a point source on $\{z=0.5\}$. True location at $(-0.5,0,0.5)^\top$ is marked by green crosses.\\
    (a) For $\p=(-17,-7,-8)^\top$, $|I(\z,\q)|$ peaks at $(-0.495, 0, 0.495)^\top$, relative error $1\%$. \\
    (b) For $\p=(17,-7-8)^\top$, $|I(\z,\q)|$ peaks at $(-0.462, -0.10, 0.375)^\top$, relative error $24.267\%$. \\
    (c) For $\p=(17,-7-8)^\top$, $\widetilde I_1(\z)$ peaks at $(-0.495, 0, 0.495)^\top$, relative error $1\%$.\\
    Here, $\z \in [-2,2]^3$, $k=20,\q =(1,1,1)^\top$. Computed locations are rounded to three decimal digits.
    }
    \label{fig:posfail}
\end{figure}
\begin{figure}[ht]
    \hspace{-0.3in}\subfloat[ $|\p \cdot \Im \G(\x,\z){\bf{e}}_1|^2$  ]
    {\begin{tikzpicture}
    \node[anchor=south west,inner sep=0] (image) at (0,0) {\includegraphics[width=4.5cm,height = 3.9 cm]{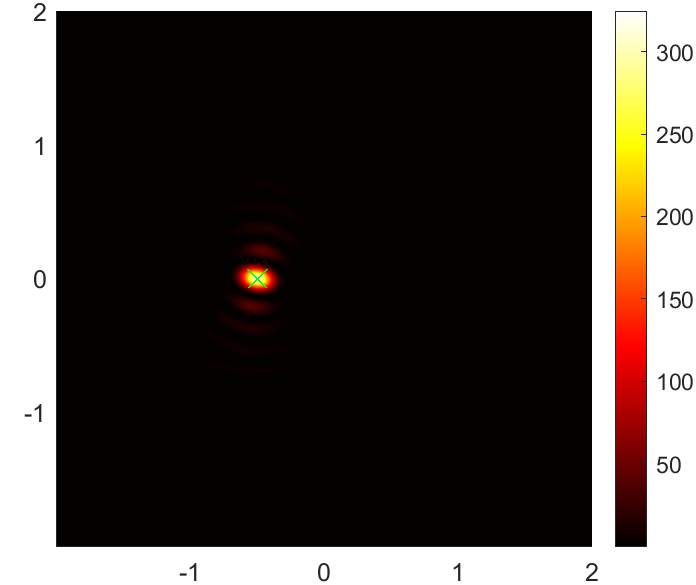}};
    \begin{scope}[x={(image.south east)},y={(image.north west)}]
      \node[anchor=north] at (.53,0) {\scriptsize{$x$}};
      \hspace{0.2cm}
      \node[anchor=south,rotate=90] at (0,.52) {\scriptsize{$y$}};
    \end{scope}
  \end{tikzpicture}  }
  \subfloat[$|\p \cdot \Im \G(\x,\z){\bf{e}}_2|^2$ ]
    {\begin{tikzpicture}
    \node[anchor=south west,inner sep=0] (image) at (0,0) {\includegraphics[width=4.5cm,height = 3.9 cm]{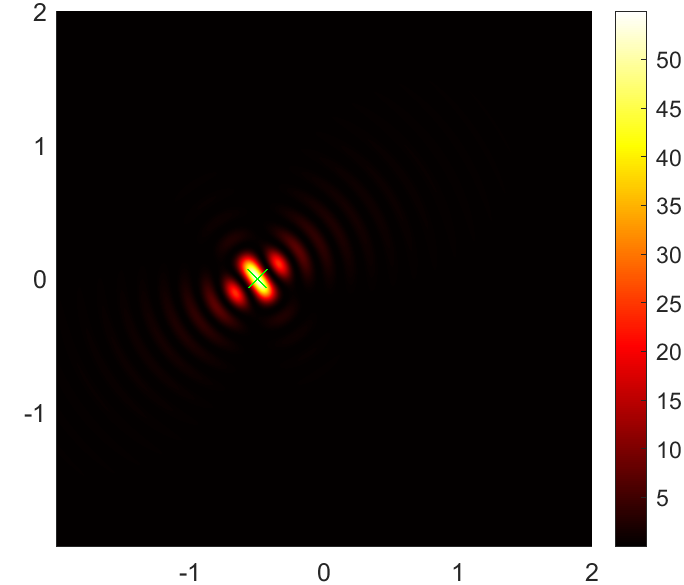}};
    \begin{scope}[x={(image.south east)},y={(image.north west)}]
      \node[anchor=north] at (.53,0) {\phantom{\scriptsize{$x$}}};
    \end{scope}
  \end{tikzpicture}  }
     \subfloat[$|\p \cdot \Im \G(\x,\z){\bf{e}}_3|^2$]
    {\begin{tikzpicture}
    \node[anchor=south west,inner sep=0] (image) at (0,0) {\includegraphics[width=4.5 cm,height = 3.9 cm]{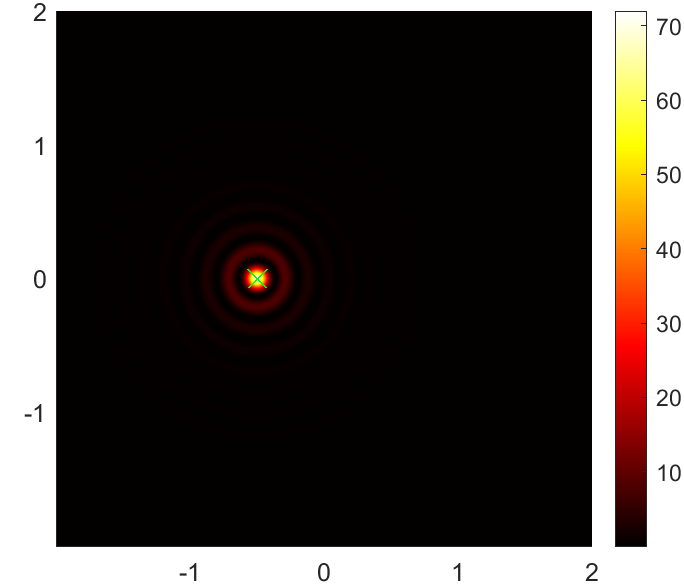}};
    \begin{scope}[x={(image.south east)},y={(image.north west)}]
      \node[anchor=north] at (.53,0) {\phantom{\scriptsize{$x$}}};
    \end{scope}
  \end{tikzpicture}  }
   \subfloat[$\widetilde I_2(\z)$ ]
    {\begin{tikzpicture}
    \node[anchor=south west,inner sep=0] (image) at (0,0) { \includegraphics[width=4.5 cm,height = 3.9 cm]{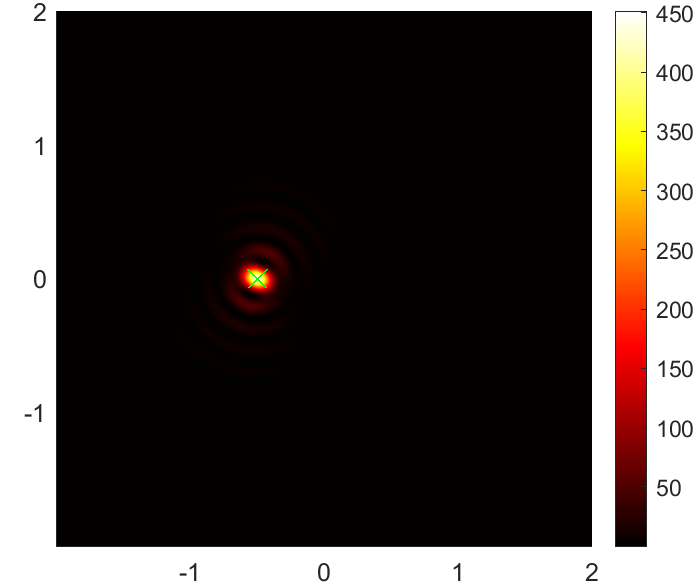}};
    \begin{scope}[x={(image.south east)},y={(image.north west)}]
      \node[anchor=north] at (.53,0) {\phantom{\scriptsize{$x$}}};
    \end{scope}
  \end{tikzpicture}  }
    \caption{ 
For $\p=(17,-7,-8)^\top$, all three terms $|\p \cdot \Im \G(\x,\z){\bf{e}}_1|^2$, $|\p \cdot \Im \G(\x,\z){\bf{e}}_2|^2$,  $|\p \cdot \Im \G(\x,\z){\bf{e}}_3|^2$ and their sum $\widetilde I_2(\z)$ attain their maximum at $\x=(-0.5,0,0.5)^\top$ in a small neighborhood of $\x$. 
    }
    \label{fig:Ie_all_peak}
\end{figure}
\begin{figure}[ht]
    \hspace{-0.3in}\subfloat[  $|\p \cdot \Im \G(\x,\z){\bf{e}}_1|^2$  ]
    {\begin{tikzpicture}
    \node[anchor=south west,inner sep=0] (image) at (0,0) {\includegraphics[width=4.5cm,height = 3.9 cm]{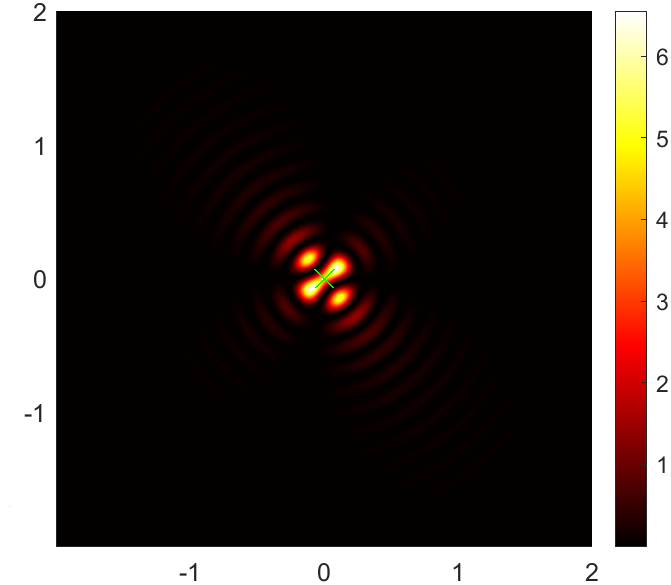}};
    \begin{scope}[x={(image.south east)},y={(image.north west)}]
      \node[anchor=north] at (.53,0) {\scriptsize{$x$}};
      \hspace{0.2cm}
      \node[anchor=south,rotate=90] at (0,.52) {\scriptsize{$y$}};
    \end{scope}
  \end{tikzpicture}  }
  \subfloat[$|\p \cdot \Im \G(\x,\z){\bf{e}}_2|^2$]
    {\begin{tikzpicture}
    \node[anchor=south west,inner sep=0] (image) at (0,0) {\includegraphics[width=4.5 cm,height = 3.9 cm]{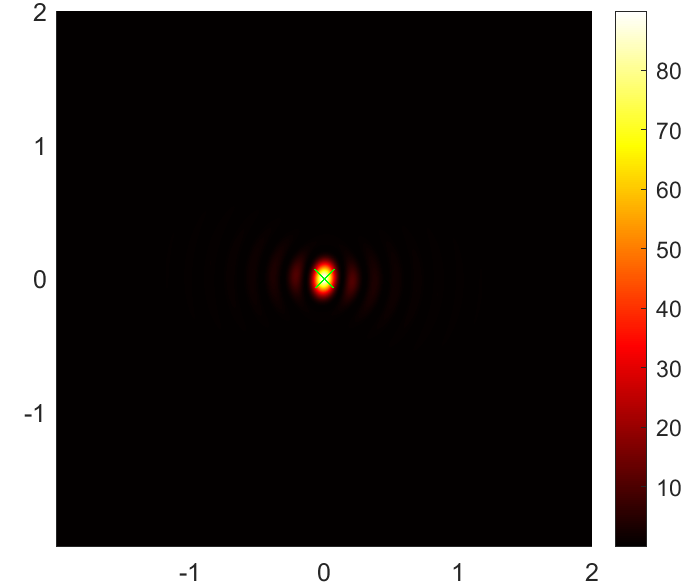}};
    \begin{scope}[x={(image.south east)},y={(image.north west)}]
      \node[anchor=north] at (.53,0) {\phantom{\scriptsize{$x$}}};
    \end{scope}
  \end{tikzpicture}  }
    \subfloat[$|\p \cdot \Im \G(\x,\z){\bf{e}}_3|^2$]
    {\begin{tikzpicture}
    \node[anchor=south west,inner sep=0] (image) at (0,0) {\includegraphics[width=4.5 cm,height = 3.9 cm]{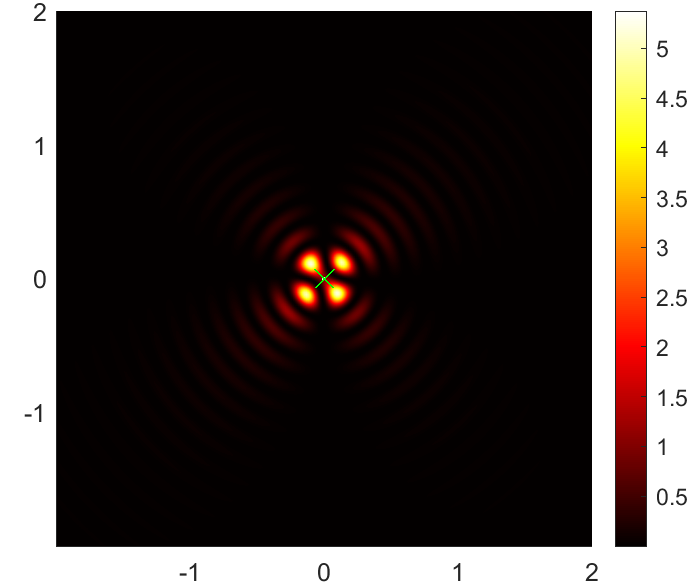}};
    \begin{scope}[x={(image.south east)},y={(image.north west)}]
      \node[anchor=north] at (.53,0) {\phantom{\scriptsize{$x$}}};
    \end{scope}
  \end{tikzpicture}  }
   \subfloat[$\widetilde I_2(\z)$]
    {\begin{tikzpicture}
    \node[anchor=south west,inner sep=0] (image) at (0,0) {\includegraphics[width=4.5 cm,height = 3.9 cm]{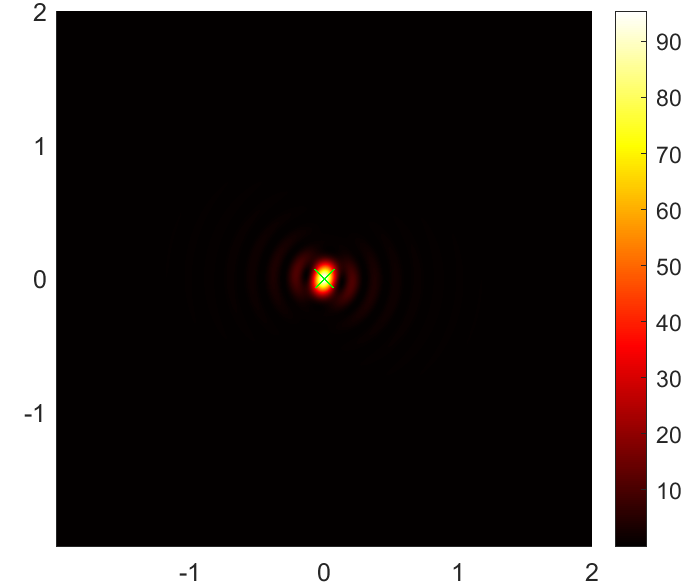}};
    \begin{scope}[x={(image.south east)},y={(image.north west)}]
      \node[anchor=north] at (.53,0) {\phantom{\scriptsize{$x$}}};
    \end{scope}
  \end{tikzpicture}  }
    \caption{ For $\p=(2,9,-1)^\top$, only one term $|\p \cdot \Im \G(\x,\z){\bf{e}}_2|^2$  attains its maximum at $\x=(0,0,0)^\top$ in a small neighborhood of $\x$, but dominates the values of the other terms that do not peak at $\x$. As a result, their sum $\widetilde I_2(\z)$ still attains a maximum at $\x$ in this neighborhood. 
    }
    \label{fig:Ie_one_peak}
\end{figure}
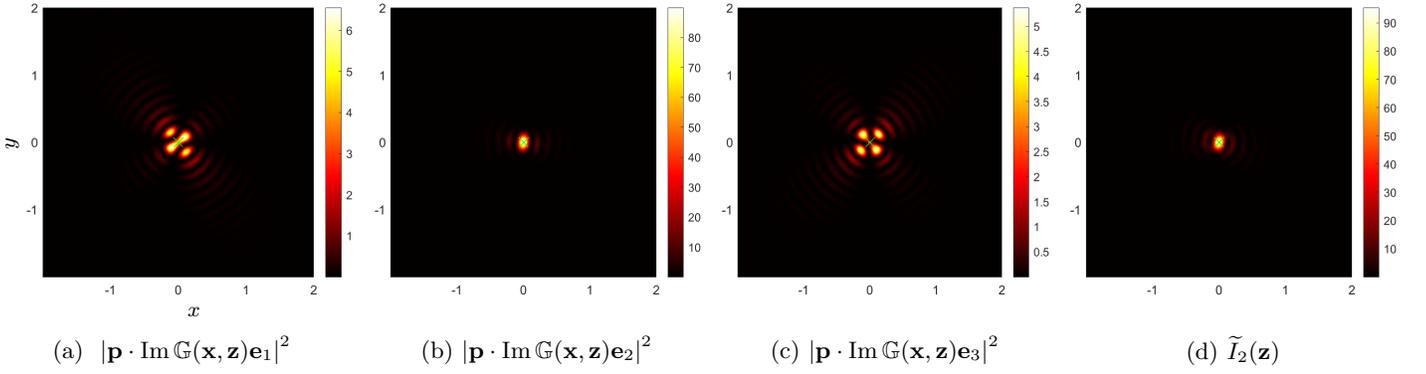
\begin{lemma}
\label{atleat1peak}
    Let  $\x\in \R^3$, and $\p=(p_1,p_2,p_3)^\top\in \R^3$ be nonzero. Then, for any  integer $s>0$, at least one of the three functions $|\p \cdot \Im \G(\x,\z){\bf{e}}_1|^s$, $|\p \cdot \Im \G(\x,\z){\bf{e}}_2|^s$,  $|\p \cdot \Im \G(\x,\z){\bf{e}}_3|^s$ attains its maximum at $\z=\x$ in a small neighborhood of $\x$.
\end{lemma}
\begin{proof}
\label{proof1}
 First, we claim that the components of $\p$ satisfy at least one of the following conditions 
\begin{align}
\label{condi}
   & 8p_1^2-p_2^2-p_3^2> 0,\notag\\
\text{ or }  & 8p_2^2-p_3^2-p_1^2> 0,\\
\text{ or } & 8p_3^2-p_1^2-p_2^2> 0.\notag
\end{align}
Indeed, assume by contradiction that
\begin{align*}
    8p_1^2-p_2^2-p_3^2\leq 0, \quad 
    8p_2^2-p_3^2-p_1^2\leq 0, \quad \text{and }
    8p_3^2-p_1^2-p_2^2 \leq 0.
\end{align*}
Adding these inequalities, we get $6(p_1^2+p_2^2+p_3^2)\leq 0,$
which contradicts the fact that $\p\in \R^3$ is nonzero. Without loss of generality, assume $ 8p_1^2-p_2^2-p_3^2> 0$,  then $p_1 \ne 0$. 
Let $\mathbf{w}=\mathbf{x}-\mathbf{z}$.   We derive from    \eqref{taylor} that 
\begin{align}
\label{pimge}
    \p \cdot \Im\G(\x,\z)\q 
    &=\frac{k}{6\pi} \p \cdot \q-\frac{k^3}{60\pi}\left(2\p \cdot \q|\bw|^2-(\p \cdot \bw)(\q\cdot \bw)\right)  + \mathcal{O}(|\bw|^4), 
\end{align}
 Letting $\q= {\bf{e}}_1$ in \eqref{pimge},  we have 
\begin{align*}
    &\p \cdot \Im\G(\x,\z){\bf{e}}_1=p_1\frac{k}{6\pi} - \frac{k^3}{60\pi}\left(2p_1|\bw|^2-(\p\cdot \bw)w_1\right)+\mathcal{O}(|\bw|^4)\notag\\
    &=p_1\frac{k}{6\pi} +\mathcal{O}(|\bw|^4)\notag\\
    &\hspace{0.5cm}- \frac{k^3}{60\pi}\left( p_1\left(2w_2^2-2\sqrt{2}w_2\frac{p_2w_1}{2\sqrt{2}p_1}+\frac{p_2^2w_1^2}{8p_1^2}\right)+p_1\left(2w_3^2-2\sqrt{2}w_3\frac{p_3w_1}{2\sqrt{2}p_1}+\frac{p_3^2w_1^2}{8p_1^2}\right)+p_1w_1^2 \frac{8p_1^2-p_2^2-p_3^2}{8p^2_1}\right)\notag\\
    &=p_1\frac{k}{6\pi} - \frac{k^3}{60\pi}p_1\left(\sqrt{2}w_2-\frac{p_2w_1}{2\sqrt{2}p_1}\right)^2 - \frac{k^3}{60\pi}p_1\left(\sqrt{2}w_3-\frac{p_3w_1}{2\sqrt{2}p_1}\right)^2-\frac{k^3}{60\pi}p_1w_1^2 \frac{8p_1^2-p_2^2-p_3^2}{8p^2_1}+\mathcal{O}(|\bw|^4).
\end{align*}    
Together with $ 8p_1^2-p_2^2-p_3^2> 0$, when $|\mathbf{w}|$ is small enough, we derive 
\begin{align*}
    &0 \leq \p \cdot \Im\G(\x,\z){\bf{e}}_1 \leq p_1\frac{k}{6\pi}, \quad \text{ if } p_1> 0,\\
    &0 \geq \p \cdot \Im\G(\x,\z){\bf{e}}_1 \geq p_1  \frac{k}{6\pi},  \quad \text{ if } p_1< 0.
\end{align*}
That leads to $$|\p \cdot \Im\G(\x,\z){\bf{e}}_1| \leq |p_1|\frac{k}{6\pi},$$ in a small neighborhood of $\x$. Moreover,  thanks to Lemma \ref{lem:imgq}, we have $$|\p \cdot \Im\G(\x,\x){\bf{e}}_1|=|p_1|\frac{k}{6\pi},$$  
so $|\p \cdot \Im\G(\x,\z){\bf{e}}_1|$ attains its maximum at $\z=\x$ when $\z$ is in a small neighborhood of $\x$. Consequently, the same result holds for $|\p \cdot \Im\G(\x,\z){\bf{e}}_1|^s$ for  any positive integer $s$. \\
Similarly, if the condition $8p_2^2-p_3^2-p_1^2> 0$ or $8p_3^2-p_1^2-p_2^2> 0$ is satisfied, $|\p \cdot \Im\G(\x,\z){\bf{e}}_2|^s$ or $|\p \cdot \Im\G(\x,\z){\bf{e}}_3|^s$ is maximal at $\x$ when $\z$ is in a small neighborhood of $\x$, respectively.
\end{proof}
We recall from Lemma \ref{atleat1peak} that the components of a nonzero vector $\mathbf{p} \in \mathbb{R}^3$ satisfy at least one of the conditions in \eqref{condi}. This leads to the fact that the corresponding product term $|\mathbf{p} \cdot \Im\G(\mathbf{x}, \mathbf{z})\mathbf{e}_i|^s$ attains a maximum at $\z=\x$ in a small neighborhood of $\mathbf{x}$ for $i \in\{ 1, 2, 3\}$. The key theorem below justifies the behavior of the imaging functions by analyzing the important properties of these products.
\begin{theorem}
\label{globaldominates}
     Let  $\x\in \R^3$, $\p=(p_1,p_2,p_3)^\top\in \R^3$ be nonzero, and $s$ be an positive integer. If the two conditions in \eqref{condi} are satisfied, say $8p_1^2-p_2^2-p_3^2>0$ and $8p_2^2-p_1^2-p_3^2>0$, then in a small neighborhood of $\x$, 
     $|\p_{} \cdot \Im \G(\x_{},\z){\bf{e}}_1|^s$ and $|\p_{} \cdot \Im \G(\x_{},\z){\bf{e}}_2|^s$ attain their maximum at $\z= \x$, and
     \begin{align}
     \label{ine:dom}
|\p_{} \cdot \Im \G(\x_{},\z){\bf{e}}_3|^s \leq \frac{|\p_{} \cdot \Im \G(\x_{},\x){\bf{e}}_1|^s+|\p_{} \cdot \Im \G(\x_{},\x){\bf{e}}_2|^s}{2^{s/2+1}}.
\end{align}
Furthermore, if only one condition in \eqref{condi} is satisfied,  say $8p_1^2-p_2^2-p_3^2>0$, then in a small neighborhood of $\x$, $|\p_{} \cdot \Im \G(\x_{},\z){\bf{e}}_1|^s$ attains its maximum at $\z= \x$ and
\begin{align}
    |\p_{} \cdot \Im \G(\x_{},\z){\bf{e}}_2|^s + |\p_{} \cdot \Im \G(\x_{},\z){\bf{e}}_3|^s\leq \frac{2}{7^{s/2}}|\p_{} \cdot \Im \G(\x_{},\x){\bf{e}}_1|^s. 
    \label{ine:dom2}
\end{align}
Consequently, in this neighborhood, the sum $$|\p_{} \cdot \Im \G(\x_{},\z){\bf{e}}_1|^s+ |\p_{} \cdot \Im \G(\x_{},\z){\bf{e}}_2|^s+ |\p_{} \cdot \Im \G(\x_{},\z){\bf{e}}_3|^s$$  achieves its maximum at $\z=\x$ for all $s\geq s_0$, where  $s_0>0$ is sufficiently large.
\end{theorem}
\begin{proof}
\label{proof2}
First, we consider the case
$$8p_1^2-p_2^2-p_3^2> 0, \quad 8p_2^2-p_1^2-p_3^2> 0 \quad \text {and} \quad 8p_3^2 - p_2^2 - p_1^2 \leq 0.$$
According to the \text{proof} of Lemma \ref{atleat1peak},
$|\p_{} \cdot \Im \G(\x_{},\z){\bf{e}}_1|^s$ and $|\p_{} \cdot \Im \G(\x_{},\z){\bf{e}}_2|^s$  attain their maximum at $\z=\x$ in a small neighborhood of $\x$ and
$$|\p_{} \cdot \Im \G(\x,\x){\bf{e}}_j|^s=\left(\frac{|p_j|k}{6\pi}\right)^s,\quad\text{ for } j=1,2.$$
Again, we derive from \eqref{taylor}
 with $\mathbf{q} = \mathbf{e}_3$ that
\begin{align}
    \p \cdot \Im\G(\x,\z){\bf{e}}_3=p_3\frac{k}{6\pi} - \frac{k^3}{60\pi}\left(2p_3|\bw|^2-(\p\cdot \bw)w_3\right)+\mathcal{O}(|\bw|^4).
    \label{imge3}
\end{align}
If $p_3=0$, then $|\p \cdot \Im\G(\x,\z){\bf{e}}_3| = \mathcal{O}(|{\bf w}|^2)$ as $|\bw| \rightarrow 0$. Thus, when $\z$ is in a small neighborhood of $\x$, 
we clearly obtain the estimate \eqref{ine:dom}. 
\\
Now, if $p_3\ne 0$, we rewrite \eqref{imge3} as
\begin{align*}
    \p \cdot \Im\G(\x,\z){\bf{e}}_3&=p_3\frac{k}{6\pi} - \frac{k^3}{60\pi}p_3\left(\sqrt{2}w_1-\frac{p_1w_3}{2\sqrt{2}p_3}\right)^2 - \frac{k^3}{60\pi}p_3\left(\sqrt{2}w_2-
    \frac{p_2w_3}{2\sqrt{2}p_3}\right)^2 \\ &-\frac{k^3}{60\pi}p_3w_3^2 \frac{8p_3^2-p_1^2-p_2^2}{8p^2_3}+\mathcal{O}(|\bw|^4).
\end{align*}
That yields, for sufficiently small $|\bw|$, 
\begin{align*}
    &0\leq \p \cdot \Im\G(\x,\z){\bf{e}}_3 \leq 
    p_3\frac{k}{6\pi} -\frac{k^3}{60\pi}p_3w_3^2 \frac{8p_3^2-p_1^2-p_2^2}{8p^2_3},\quad \text{ if } p_3>0,\\
     &0\geq \p \cdot \Im\G(\x,\z){\bf{e}}_3 \geq 
    p_3\frac{k}{6\pi} -\frac{k^3}{60\pi}p_3w_3^2 \frac{8p_3^2-p_1^2-p_2^2}{8p^2_3},\quad \text{ if } p_3<0.
\end{align*}
Recall that $8p_3^2 - p_2^2 - p_1^2 \leq 0$. Denote $a=p_1^2+p_2^2$ and  $f(t)= t\frac{k}{6\pi} -t\frac{k^3w_3^2}{60\pi}+\frac{1}{8t}\frac{ak^3w_3^2}{60\pi}$ with $t\ne 0$, then 
\begin{align}
    &0\leq \p \cdot \Im\G(\x,\z){\bf{e}}_3 \leq 
    f(p_3), \quad \text{ if } 0<p_3\leq \sqrt{\frac{a}{8}},\label{pimge3}\\
     &0\geq \p \cdot \Im\G(\x,\z){\bf{e}}_3 \geq 
    f(p_3), \quad \text{ if } -\sqrt{\frac{a}{8}}\leq p_3<0 .\label{pimge31}
\end{align}
Moreover, when $|\bw|$ is small enough, $$f'(t) = \frac{k}{6\pi}\left(1-\frac{k^2}{10}w^2_3\left(1+\frac{a}{8t^2}\right)\right)>0.$$
Thus, $f(t)$ is a strictly increasing function on  $(-\infty,0)$ and $(0,+\infty)$ when $\z$ is in a small neighborhood of $\x$. That leads to 
\begin{align}
    & f(p_3)\leq f\left(\sqrt{\frac{a}{8}}\right) =  \sqrt{\frac{p_1^2+p_2^2}{8}}\frac{k}{6\pi},\quad \text{ if } 0<p_3\leq \sqrt{\frac{a}{8}},\label{fp3}\\
    &f(p_3)\geq f\left(-\sqrt{\frac{a}{8}}\right) =  -\sqrt{\frac{p_1^2+p_2^2}{8}}\frac{k}{6\pi},\quad \text{ if } -\sqrt{\frac{a}{8}}\leq p_3<0.\label{fp31}
\end{align}
Combine together the inequalities \eqref{pimge3}--\eqref{fp31}, we get $$|\p \cdot \Im\G(\x,\z){\bf{e}}_3 | \leq \sqrt{\frac{p_1^2+p_2^2}{8}}\frac{k}{6\pi} \leq {\frac{|p_1|+|p_2|}{\sqrt{8}}}\frac{k}{6\pi}
={\frac{|\p \cdot \Im\G(\x,\x){\bf{e}}_1 |+|\p \cdot \Im\G(\x,\x){\bf{e}}_2 |}{\sqrt{8}}},$$
in a small neighborhood of $\x$. Then, using the fact that for any $c,d\geq 0$ and integer $s>0$,
$$(c+d)^s\leq 2^{s-1}(c^s+d^s),$$  we obtain estimate \eqref{ine:dom}. 
\\
Next, we consider the case $$8p_1^2-p_2^2-p_3^2>0,\quad 8p_2^2-p_1^2-p_3^2\leq 0 \quad \text{and} \quad 8p_3^2-p_1^2-p_2^2\leq 0.$$ 
Again, by the \text{proof} of Lemma \ref{atleat1peak}, 
$|\p_{} \cdot \Im \G(\x_{},\z){\bf{e}}_1|^s$  attain its maximum at $\z=\x$  in a small neighborhood of $\x$ and $$|\p_{} \cdot \Im \G(\x,\x){\bf{e}}_1|^s =\left(\frac{|p_1|k}{6\pi}\right)^s.$$
We also derive from \eqref{taylor}
 with $\mathbf{q} = \mathbf{e}_2$ that
\begin{align*}
    \p \cdot \Im\G(\x,\z){\bf{e}}_2=p_2\frac{k}{6\pi} - \frac{k^3}{60\pi}\left(2p_2|\bw|^2-(\p\cdot \bw)w_2\right)+\mathcal{O}(|\bw|^4).
\end{align*}
If $p_2= 0$, then $|\p \cdot \Im\G(\x,\z){\bf{e}}_2| = \mathcal{O}(|{\bf w}|^2)$ as $|\bw| \rightarrow 0$. Thus, in a small neighborhood of $\x$, we clearly have $ |\p_{} \cdot \Im \G(\x_{},\z){\bf{e}}_2|^s \leq \frac{1}{7^{s/2}}|\p_{} \cdot \Im \G(\x_{},\x){\bf{e}}_1|^s.$
\vspace{0.1cm}\\
Similarly, if $p_2\ne 0$, for sufficiently small $|\bw|$, we obtain
\begin{align*}
    &0\leq \p \cdot \Im\G(\x,\z){\bf{e}}_2 \leq 
    p_2\frac{k}{6\pi} -\frac{k^3}{60\pi}p_2w_2^2 \frac{8p_2^2-p_1^2-p_3^2}{8p^2_2},\quad \text{ if } p_2>0,\\
     &0\geq \p \cdot \Im\G(\x,\z){\bf{e}}_2 \geq 
    p_2\frac{k}{6\pi} -\frac{k^3}{60\pi}p_2w_2^2 \frac{8p_2^2-p_1^2-p_3^2}{8p^2_2},\quad \text{ if } p_2<0.
\end{align*}
Recall that $8p_2^2 - p_1^2 - p_3^2 \leq 0$. Let $b=p_1^2+p_3^2$ and $g(t)= t\frac{k}{6\pi} -t\frac{k^3w_2^2}{60\pi}+\frac{1}{8t}\frac{bk^3w_2^2}{60\pi}$ with $t\ne 0$. Again,  \vspace{0.1cm}\\
we can show that $g$ is strictly increasing on $(-\infty,0)$ and $(0,+\infty)$ when $|\bw|$ is small enough. 
Thus,  
\begin{align}
\label{mainineq}
&|\p \cdot \Im\G(\x,\z){\bf{e}}_2 | \leq \sqrt{\frac{p_1^2+p_3^2}{8}}\frac{k}{6\pi},
\end{align}
within this neighborhood. Additionally, the inequalities $p_2^2\leq \frac{p^2_1+p_3^2}{8}$ and $p_3^2\leq \frac{p^2_1+p_2^2}{8}$ imply that  $p_3^2\leq \frac{1}{7}p_1^2$. Together with $|\p \cdot \Im\G(\x,\x){\bf{e}}_1 |=|p_1|\frac{k}{6\pi}$, and for any integer $s>0$, we derive from \eqref{mainineq} 
\begin{align}
    |\p_{} \cdot \Im \G(\x_{},\z){\bf{e}}_2|^s \leq \frac{1}{7^{s/2}}|\p_{} \cdot \Im \G(\x_{},\x){\bf{e}}_1|^s, \label{estimateforp2}
\end{align}
in a small neighborhood of $\x$. Similarly, we obtain 
\begin{align*}
    |\p_{} \cdot \Im \G(\x_{},\z){\bf{e}}_3|^s \leq \frac{1}{7^{s/2}}|\p_{} \cdot \Im \G(\x_{},\x){\bf{e}}_1|^s
\end{align*}
which, together with  estimate \eqref{estimateforp2}, implies \eqref{ine:dom2}.
\end{proof}
\begin{remark}
\label{remarkrealp}
    We rewrite the imaging function $\widetilde {I_s}(\z)$ using Lemma \ref{theo1} as 
    $$\widetilde {I_s}(\z)=\sum_{i=1}^3\left| \sum_{j=1}^N  \p_j \cdot \Im \G(\x_j,\z){\bf{e}}_i
\right|^s.$$ 
Under the well-separated sources assumption \eqref{assump:well-sep}, we have $ \Im \G(\x_j,\x_i)\approx 0$ for $i\ne j$. If all real moment vectors $\p_j$ have comparable magnitudes, then in a small neighborhood of $\x_j$ for $j=1,2,\hdots,N$, 
\begin{align*}
    \widetilde {I_s}(\z)\approx \sum_{i=1}^3\left|    \p_j \cdot \Im \G(\x_j,\z){\bf{e}}_i
\right|^s.
\end{align*}
By applying Theorem \ref{globaldominates} with $\p=\p_j$ and $\x=\x_j$, $\widetilde {I_s}(\z)$  has a maximum at $\z=\x_j$ in this neighborhood for some integer $s>0$. 
In addition, from Remark \ref{re:Idecay}, we obtain that $\widetilde I_s(\z)$ decays quickly as $\z$ moves away from the sources  up to the leading order
$$ \widetilde {I^\text{}_s}(\z)= \mathcal{O} \Big ( \text{dist}(\z,{\bf X})^{-s} \Big ), \quad \text{as} \quad \text{dist}(\z,{\bf X}) \rightarrow \infty,$$ 
where  ${\bf X} = \{\x_j : j=1, 2, \ldots, N\}$. 
Thus, $\widetilde {I^\text{}_s}(\z)$ is expected to have peaks at $\z=\x_j$ for comparable $|\p_j|$. We then determine $N$ based on the number of peaks and $\x_j$ by the locations of these peaks. Conversely, if some vectors $ \p_j$ significantly exceed others in magnitude, the value of $\widetilde {I^\text{}_s}(\z)$ at these corresponding sources will dominate those of the other sources, resulting in significant peaks only at these locations.
This poses a challenge to detecting sources with notably smaller $|\p_j|$.
\end{remark}
To address this, we present a numerical algorithm for fast and accurate identification of sources with real moment vectors whose magnitudes are not comparable. The main idea is to iteratively update the imaging functions, turning cases where sources differ notably in magnitude to comparable ones. Here is a detailed explanation of the algorithm.
\vspace{0.3cm}


We assume without the loss of generality that $|\p_1|\geq |\p_2|\geq \hdots\geq |\p_N|$. First, as mentioned earlier,  
$\widetilde {I^\text{}_s}(\z)$ for some integer $s>0$ has significant peaks at $\x_j$ with $| \p_j|$ dominating those of the other sources and are comparable to each other for $j=1,\hdots, t_1$, where $t_1 \in \{1,2,\hdots,N\}$. These peaks allow us to identify these source locations. 
We then compute their moment vectors. Again, under the assumption of well-separated sources \eqref{assump:well-sep}, $\Im\G(\x_j,\x_l)\q\approx 0$ for any $j\ne l$. By Lemma \ref{theo1} and the fact that $|\p_j|$ dominates the magnitude of the other moment vectors,
we can estimate
$$  I(\x_{j},\q)\approx    \p_{j}\cdot \Im\G(\x_{j},\x_{j})\q \quad \text{ for } j=1,\hdots, t_1,$$
for a nonzero vector $\q\in\R^3$. Thanks to Lemma \ref{lem:imgq}, we can choose  $$\q = [\Im\G(\x_{j},\x_{j})]^{-1}{\bf{e}}_i=\frac{6\pi}{k}{\bf{e}}_i \quad \text{ for } i=1,2,3.$$ 
With this choice, the $i$-th component $(  \p_{j})_i$ of the moment vector $  \p_{j}$ can be computed as 
$$(  \p_{j})_i\approx   I(\x_{j},\q) \quad \text{ for } j=1,\hdots, t_1\text{ and } i=1,2,3.$$ 
Here, $I(\x_{j},\q)$ is computed by the definition in \eqref{eq:im} provided the Cauchy measurements, an approximation of $\x_{j}$ and $\q = \frac{6\pi}{k}\mathbf{e}_i$.
\\
Next, we update the imaging function $\widetilde {I^\text{}_s}(\z)$ as follows
\begin{align*}
    \widetilde {I}_{s,1}(\z) &:= \sum_{i=1}^3\left|  I(\z,{\bf{e}}_i)-\sum_{j=  1}^{  {t_1}}  \p_{j} \cdot \Im \G(\x_{j},\z)\mathbf e_i\right|^s\\
    &:= \sum_{i=1}^3\left|  I_1(\z,{\bf{e}}_i)\right|^s.
\end{align*}
This way, we filter out dominant terms identified explicitly in the previous step and obtain 
$$ \widetilde {I}_{s,1}(\z) = \sum_{i=1}^3\left|\sum_{j=  {t_1+1}}^{  {N}}  \p_{j} \cdot \Im \G(\x_{j},\z)\mathbf e_i\right|^s.$$ 
We see that this updated imaging function retains a structure similar to $\widetilde {I^\text{}_s}(\z)$ and therefore exhibits similar behavior. Suppose that among the remaining sources, those at $\x_j$ for $j=t_1+1,\hdots t_2$, where $t_2\in \{1,2,\hdots,N\}$, have $|  \p_j|$ that are notably larger than those of the other sources and comparable among themselves. 
Then, the updated imaging function $\widetilde {I}_{s,1}(\z)$ has peaks at these $\x_j$ for some integer $s>0$, providing estimations for these locations. Moreover, the $i$-th component $(  \p_{j})_i$ of  $  \p_{j}$ can be computed as 
$$(  \p_{j})_i\approx    I(\x_{j},\q)- \sum_{l=  {1}}^{  {t_1}}   \p_{l} \cdot \Im \G(\x_{j},\x_{l})\q = I_1(\x_j,\q) \quad \text{ for } j=t_1+1,\hdots, t_2\text{ and } i=1,2,3, $$  
where $\q = \frac{6\pi}{k}\mathbf{e}_i$. Again, this is because of the domination of 
$|  \p_j|$ and the assumption of well-separated sources \eqref{assump:well-sep}. 
Here, $I(\x_{j},\q)$ is computed by \eqref{eq:im} based on the Cauchy measurements, an approximation of $\x_{j}$ and $\q = \frac{6\pi}{k}\mathbf{e}_i$.
\vspace{0.1cm}
\\
We repeat this process to determine the other sources. The general formula of the updated imaging function to determine sources at $\x_j$ with $|  \p_j|$  dominating those of the remaining sources and are comparable to each other for $j=t_{n}+1, \hdots, t_{n+1}$, where $t_{n+1}\in \{1,2,\hdots,N\}$, is given by \begin{align*} 
  \widetilde {I}_{s,n}(\z) &:= \sum_{i=1}^3\left|  I_{n-1}(\z,\mathbf{e}_i) -\sum_{j=  {t_{n-1}+1}}^{  {t_{n}}}   \p_{j} \cdot \Im \G(\x_{j},\z)\mathbf{e}_i \right|^s\ := \sum_{i=1}^3\left|  I_n(\z,{\bf{e}}_i)\right|^s.\\
&= \sum_{i=1}^3\left|\sum_{j=  {t_{n}+1}}^{  N}   \p_{j} \cdot \Im \G(\x_{j},\z)\mathbf{e}_i\right|^s \quad \quad \text{ for } n\in \N, \text{ } n\geq 1 \text{ and } I_0:=I. 
\end{align*} 
The algorithm is summarized as follows.
\vspace{0.3cm}
\\
\textbf{Algorithm.}
\begin{itemize}
    \item \textit{ Step 1:} 
\begin{itemize}
        \item  Locate  sources at $\x_j$ for $j=1,\hdots,t_1$ using $\widetilde {I_s}(\z)$  for some integer $s>0$.
        \item Estimate the $i$-th component $( \p_{j})_i$ of moment vector $\p_j$ by $  I(\x_{j},\q)$ with $\q = \frac{6\pi}{k}{\bf{e}}_i$ for $i=1,2,3$ and $j=1,\hdots,t_1$.
    \end{itemize}    
    \item \textit{ Step $n$:} 
    \begin{itemize}
        \item Update the imaging function 
    \begin{align*}
    \widetilde {I}_{s,n}(\z) &:= \sum_{i=1}^3\left|  I_{n-1}(\z,{\bf{e}}_i)-\sum_{j=t_{n-1}+1}^{t_n}   \p_{j} \cdot \Im \G(\x_{j},\z)\mathbf e_i\right|^s = \sum_{i=1}^3|  I_n(\z,{\bf{e}}_i)|^s.
\end{align*} 
        \item Locate sources at $\x_j$ using $ \widetilde {I}_{s,n}(\z)$ for some integer $s>0$, and
        estimate $(  \p_{j})_i$ by $  I_n(\x_{j},\q)$ with $\q = \frac{6\pi}{k}{\bf{e}}_i$ for $i=1,2,3$ and $j=t_{n}+1, \hdots, t_{n+1}$.
    \end{itemize}    
\item  Stop when the updated imaging function no longer exhibits significant peaks.
\end{itemize}
     
    
Computing $\widetilde {I_s}(\z)$ at step $1$ is simple and fast. The updates of $\widetilde {I}_{s,n}(\z)$ in the next steps only need to evaluate the terms $\p_{j} \cdot \Im \G(\x_{j},\z)\q$, which can also be done quickly. 
\vspace{0.3cm}
{
\begin{remark}
In our numerical study, we observed that when 
 the difference in the magnitude of moment vectors between the strongest and weakest sources
is about six times or greater, the proposed algorithm may encounter a challenge that the updated imaging function at some step may inherit a major error from the previous step. The primary source of error in this case arises from a small neighborhood around a peak, which indicates the location of a source. In other words, the subtraction technique used in the algorithm does not 
totally
eliminate the small neighborhood around a peak. While this error may not significantly impact the identification of relatively strong sources, it can pose a substantial challenge when attempting to identify much weaker sources in subsequent steps.

One possible solution to address this issue is to ``clean up" the small neighborhoods (e.g., by applying suitable cutoff functions) around the sources identified in the previous step before proceeding to the next one in the algorithm. This clean-up process can help eliminate the residual errors after each step and is a reasonable approach, given our assumption that the sources are well-separated. 
\end{remark}
}

Next, we extend these results to reconstruct point sources with complex moment vectors. We propose two new imaging functions as follows
\begin{align}
&\widetilde {I}^\text{re}_s(\z) := |\Re I(\z,{\bf{e}}_1)|^s +  |\Re I(\z,{\bf{e}}_2) |^s +  |\Re I(\z,{\bf{e}}_3)|^s,\label{Icomplex1} \\
&
 \widetilde {I}^\text{im}_s(\z) := |\Im I(\z,{\bf{e}}_1) |^s +  |\Im I(\z,{\bf{e}}_2) |^s +  |\Im(I(\z,{\bf{e}}_3) |^s.
 \label{Icomplex}
\end{align}
Note that taking the real and imaginary part of the identity in Lemma~\ref{theo1}  yields  
\begin{align*}
    &\Re I(\z,{\bf{e}}_i) = \sum_{j=1}^N \Re \left(\p_j \cdot \Im \G(\x_j,\z){\bf{e}}_i \right)=\sum_{j=1}^N \Re \p_j \cdot \Im \G(\x_j,\z){\bf{e}}_i,\\
    &\Im I(\z,{\bf{e}}_i) = \sum_{j=1}^N \Im \left(\p_j \cdot \Im \G(\x_j,\z){\bf{e}}_i \right)=\sum_{j=1}^N \Im \p_j \cdot \Im \G(\x_j,\z){\bf{e}}_i,
    \quad \text{for} \quad i=1,2,3.
\end{align*}
 Since $\Re\p_j$ and $\Im \p_j$ are  in $\R^3$, they serve the same role as the real moment vectors $\p_j$ in the formula of $\widetilde I_s(\z)$ from \eqref{tildeI}--\eqref{imzei}. 
 Consequently, the analysis for $\widetilde I_s(\z)$ can be similarly applied to the two imaging functions defined in \eqref{Icomplex1} and \eqref{Icomplex}. In other words, as in Remark \ref{remarkrealp}, $\widetilde {I}^\text{re}_s(\z)$ exhibits significant peaks at $\z=\x_j$ for some integer $s>0$ and $j=1,2,\hdots,N$, when all $|\Re \p_{j}|$ are comparable. Similarly, 
  $\widetilde {I}^\text{im}_s(\z)$ displays significant peaks at $\z=\x_j$ for some integer $s>0$, when all $|\Im \p_{j}|$ are comparable. 
Furthermore, performing the same algorithm that we previously established for $\widetilde {I^\text{}_s}(\z)$, we can iteratively update the imaging function  $\widetilde {I}^\text{re}_s(\z)$ to find $\x_j$ and $\Re \p_{j}$. Similarly, $\Im \p_{j}$ (and also $\x_j$) is  computed using  the algorithm for $\widetilde {I}^\text{im}_s(\z)$.

\vspace{0.1cm} 

The processes of determining $(\x_j, \Re {\p_j}) $ and $ (\x_l, \Im {\p_l}) $ are done independently with different descending orders of $ |\Re \p_j| $ and $ |\Im \p_l| $.
When a real and an imaginary part share a similar computed location, we combine them to achieve a final reconstruction of the corresponding moment vector. If a computed location only appears from the process of finding either $\Re {\p_j}$  or $\Im {\p_j}$, the corresponding moment vector is then respectively determined as either $\Re {\p_j}$  or $\Im {\p_j}$.

\vspace{0.3 cm}

This method is robust against the presence of noise in data. In practice, data is always perturbed by noise. We assume that the noisy Cauchy measurements satisfy
\begin{align*}
   & \|\E(\x)-\E_\delta(\x)\|_{L^2(\partial\Omega)}\leq \delta_1\|\E(\x)\|_{L^2(\partial\Omega)},\\
   & \|\curl \E(\x)\times \bm{\nu} -\curl\E_\delta(\x)\times \bm{\nu} \|_{L^2(\partial\Omega)}\leq \delta_2\|\curl \E(\x)\times \bm{\nu}\|_{L^2(\partial\Omega)},
\end{align*}
for some constants $\delta_1,\delta_2>0$ presenting the level of noise. We can derive the following stability estimates. 
\begin{theorem} Denote by $I^\delta(\z,\q)$ the function $I(\z,\q)$ with the noisy Cauchy data. Then, for any sampling point $\z\in \R^3$, a fixed integer $s >0$, and a fixed nonzero vector $\q \in \R^3$,
\begin{align*}
    \left| \left|I(\z,\q)\right|^s-|I^\delta(\z,\q)| ^s\right| \leq C_1\max{(\delta_1,\delta_2)},\quad \text{ as } \max{(\delta_1,\delta_2)}\rightarrow 0,
\end{align*}
for some constant $C_1>0$ independent of $\delta_1,\delta_2$ and $\z$. As a result, let  $\widetilde I^\delta_{s}(\z)$ be the noisy version of $\widetilde I_s(\z)$,
\begin{align*}
  \left|\widetilde I_s(\z)-\widetilde {I}^\delta_{s}(\z)\right|\leq C_2\max{(\delta_1,\delta_2)},\quad \text{ as } \max{(\delta_1,\delta_2)}\rightarrow 0,
  \label{stabesti}
\end{align*}
for some constant $C_2>0$ independent of $\delta_1,\delta_2$ and $\z$. We also obtain similar stability estimates  for the imaging functions $\widetilde {I}^{\text{re},\delta}_{s}(\z)$ and $\widetilde {I}^{\text{im},\delta}_{s}(\z)$. 
\end{theorem}
\begin{proof}
   The proof can be done similarly as in \cite{Nguyen2023}, thus we omit it here.
\end{proof}
\vspace{0.1cm}

To conclude this section, it is worth noting that the imaging functions developed from   $I(\z,\q)$ defined in \eqref{eq:im} require no specific restrictions on the sampling points $\z$. This enables imaging at an arbitrary distance from the measurement boundary, including areas close to it. When the sampling points are distinct from those on the measurement boundary, we discuss the following base function $ \widehat{I}(\z,\q)$, which can be employed similarly as $I(\z,\q)$ for source reconstruction.
\begin{remark}  Assume that $\z$ is distinct from points on the boundary measurement $\x\in \partial \Omega$. 
We introduce the function 
\begin{equation}
\label{eq:Ihat}
    \widehat{I}(\z,\q) := \frac{i}{2} \int_{\partial \Omega} \curl (\overline{\G}(\x,\z)\q)\times \bm{\nu}\cdot \E(\x)-\curl \E(\x) \times \bm{\nu}\cdot\overline{\G}(\x,\z)\q \text{ }\d s(\x).
\end{equation}
    Similar to Lemma \ref{theo1}, we also can prove that 
    \begin{align}
    \widehat{I}(\z,\q) =\sum_{j=1}^N \p_j\cdot \Im\G(\x_j,\z)\q,\label{im2_id}
\end{align} using the following identity of the Green's tensor 
    \begin{equation*}
    \int_{\partial \Omega} \curl (\overline{\G}(\x,\z)\q)\times \bm{\nu} \cdot \G(\x,\y)\p-\curl (\G(\x,\y)\p)\times \bm{\nu}\cdot \overline{\G}(\x,\z)\q \text{ }\d s(\x) = -2i\p \cdot \Im\G(\y,\z)\q,
    \label{eq:Iz_g}
\end{equation*}    
for any $\y\in\Omega$, $\z \in \R^3$, and $\z \notin \partial \Omega $  to avoid singularity.
By combining \eqref{im2_id} with Theorem \ref{theo1}, it follows that $\widehat{I}(\z,\q)$ is expected to behave similarly to ${I}(\z,\q)$. We then denote another imaging function similar to \eqref{tildeI} using the base $\widehat{I}(\z,\q)$ by 
\begin{align}
\widetilde{\widehat{I}_s}(\z) = |\widehat{I}(\z,{\bf{e}}_1)|^s +  |\widehat{I}(\z,{\bf{e}}_2)|^s +  |\widehat{I}(\z,{\bf{e}}_3)|^s.
\label{tildehatI}
\end{align}
Similarly, we can also extend the imaging functions in \eqref{Icomplex1}--\eqref{Icomplex}, and the numerical algorithm for $\widehat{I}(\z,\q)$.

\end{remark}
\section{\textbf{Identification of  electromagnetic sources  with small volumes }}
\label{se:smallballs}
In this section, we extend our method to determine small-volume electromagnetic sources. We consider sources having compact support within $N\in \N$ small-volume domains $D_j$, expressed as $$\F=\sum_{j=1}^N \p_j{{1}}_{D_j},$$
where ${1}_{D_j}$ is the indicator function on $D_j$ and 
$D_j = \x_j + \epsilon B_j \subset \Omega$.
We assume that $D_j \cap D_i = \emptyset$ for $ i \neq j$ and that the parameter $\epsilon>0$ is small.  The bounded domains $B_j$ have Lipschitz boundaries and contain the origin, with  $|B_j| = \mathcal{O}(1)$. Denote $D=\cup_{j=1}^ND_j$, then $|D|=\mathcal{O}(\epsilon^3)$. Each $\x_j\in \R^3$ is located in  $D_j$, and is assumed to be within the measurement domain $\Omega$ with  $\text{dist}_{i\ne j}(\x_i , \x_j) \gg \lambda$. Additionally, each $D_j$ is characterized by a nonzero vector $\p_j\in\mathbb{C}^3$. 


Now, the electric field $\E$ that is generated by the source $\F$ satisfy  
\begin{equation}
\label{small-sourceprob}
\curl \curl \E (\x)- k^2 \E (\x)= \F , \quad  \text{ in } \mathbb{R}^3.
\end{equation}
As in the previous section, we assume that the electric field $\E$ is radiating, meaning that it satisfies the 
Silver-Müller radiation condition \eqref{eq:silver}. Again, the radiation condition holds uniformly with respect to $\x/|\x|$. It is well-known that the radiated field $\E$ is given by 
\begin{equation}\label{efield-smallsource}
\E (\x) =  \int_{D}  \G(\x,\y) \F(\y) \d \y,
\end{equation} 
where again $\G(\x,\y)$ is the  Green's tensor in \eqref{greentensor}.

Our goal is to solve the inverse problem of determining the number of unknown small-volume sources $N$, points $\x_j$ within $ D_j\subset \Omega$, and the directions of vectors $\p_j$ for $j=1,2,\hdots, N$. This is based on the Cauchy data measurements of
$\E(\x)$ and $\text{curl} \, \E(\x) \times \bm{\nu}$  for all $\x \in {\partial \Omega}$,
which take the forms
$$\E (\x) \big|_{{\partial \Omega}}  =  \int_{D}  \G(\x,\y) \F(\y) \,\text{d}\y \big|_{{\partial \Omega}}  \quad \text{ and } \quad \text{curl} \, \E (\x) \big|_{{\partial \Omega}}  = \text{curl} \int_{D}  \G(\x,\y) \F(\y) \,\text{d}\y\big|_{{\partial \Omega}} .$$
 Recall that ${\partial \Omega}$ denotes the closed known measurement surface, with $\bm{\nu}$ as the unit outward normal vector. To proceed, we assume $\text{dist}({\partial \Omega}, D)>0$, which implies that the Green's tensor is a smooth function in the variable $\y \in D$. Given this, we can perform a component-wise Taylor expansion in $D_j$ such that 
\begin{equation}\label{tensor-taylor}
\G(\x,\y) =  \G(\x,\x_j) + \mathcal{O}(\epsilon).
\end{equation} 
This is because $\y = \x_j + \epsilon \mathbf w$ for some $\mathbf w \in B_j$. By incorporating \eqref{tensor-taylor} into the integral formulas of $\E(\x)$ and $\curl \E(\x)$ mentioned above, we obtain their representation as follows
\begin{equation} 
\E (\x) = \sum\limits_{j=1}^{N} \epsilon^3 |B_j| \G(\x,\x_j) \text{Avg}(\F_j) +\mathcal{O}(\epsilon^4) \label{field-asym}
\end{equation}
and 
\begin{equation} 
 \text{curl} \, \E (\x) = \sum\limits_{j=1}^{N} \epsilon^3 |B_j| \text{curl} \big(\G(\x,\x_j) \text{Avg}(\F_j) \big)+\mathcal{O}(\epsilon^4) ,\label{curl-asym}
\end{equation}
for all $\x \in {\partial \Omega}$. Here, we let $\text{Avg}(\F_j)$ denote the average value of $\F$ on the subregion $D_j$. Moreover, 
$$\text{Avg}(\F_j) = \frac{1}{|D_j|} \int_{D_j} \F(\y)d\y 
=\p_j  \quad  \text{for all} \quad j = 1,2, \hdots , N.$$ 
Notice that in  \eqref{field-asym}--\eqref{curl-asym}, we use the fact that $|D_j| = \epsilon^3 |B_j|$. This means that the electromagnetic data for the case of a source with small-volume support is up to the leading order and has the same structure as the electromagnetic data given by finite point sources. 
\vspace{0.3cm}

Now, we recall the proposed base function $I(\z,\q)$
\begin{align*}
    I(\z,\q) =\int_{\partial \Omega} \curl (\Im{\G}(\x,\z)\q)\times \bm{\nu} \cdot \E(\x)-\curl \E(\x) \times \bm{\nu}\cdot\Im{\G}(\x,\z)\q\,\text{d} s(\x),
\end{align*}
for any sampling point $\z \in \R^3$ and a fixed nonzero vector $\q \in \R^3$. By using to the asymptotic expressions \eqref{field-asym}--\eqref{curl-asym} and Lemma \ref{theo1}, we rewrite the function as 
\begin{align*}
I(\z,\q) &= \sum\limits_{j=1}^{N} \epsilon^3  |B_j| \int_{{\partial \Omega}}  \text{curl} \, \big( \Im{\G}(\x,\z)\q\big)\times \bm{\nu} \cdot \G(\x,\x_j) \p_j \\
&\hspace{1.2in} -   \text{curl} \, \left(\G(\x,\x_j)\p_j \right)\times \bm{\nu}  \cdot \Im{\G}(\x,\z)\q\,\text{d}s(\x) + \mathcal{O}(\epsilon^4)\\
&= \sum_{j=1}^N\epsilon^3  |B_j|  \p_j\cdot \Im\G(\x_j,\z)\q + \mathcal{O}(\epsilon^4).
\end{align*}
Based on this, up to the leading order, $I(\z,\q)$ is expected to behave similarly to the case of point sources discussed in Section \ref{se:pointsource}, with each point source at $\x_j$ having the moment vector $\epsilon^3 |B_j|\p_j$ for $j=1,2,\hdots, N$. 
Therefore, we can use the two new imaging functions $\widetilde {I}^\text{re}_s(\z)$ and $ \widetilde {I}^\text{im}_s(\z)$, along with their updated versions following the proposed algorithm in the previous section, to determine points $\x_j$ within the small-volume subregions. Recall that
\begin{align*}
&\widetilde {I}^\text{re}_s(\z) = |\Re I(\z,{\bf{e}}_1)|^s +  |\Re I(\z,{\bf{e}}_2) |^s +  |\Re I(\z,{\bf{e}}_3)|^s, \\
&
 \widetilde {I}^\text{im}_s(\z) = |\Im I(\z,{\bf{e}}_1) |^s +  |\Im I(\z,{\bf{e}}_2) |^s +  |\Im(I(\z,{\bf{e}}_3) |^s.
 \label{Icomplex}
\end{align*}
 Since, again, these imaging functions peak as sampling point $\z$ approaches $\x_j$ for some integer $s>0$ and decay quickly as the sampling point $\z$ moves away from the small-volume sources. Specifically, for $\z \in \R^3$,
$$\widetilde {I}^\text{re}_s(\z)=\widetilde {I}^\text{im}_s(\z) = \mathcal{O} \Big ( \text{dist}(\z,{\bf X})^{-s} \Big ), \quad \text{as} \quad \text{dist}(\z,{\bf X}) \rightarrow \infty,$$ 
 where the set ${\bf X} = \{\x_j : j=1, 2, \ldots, N\}$. Furthermore, following the numerical algorithm, 
 we can compute and normalize the terms $\epsilon^3  |B_j| \p_j$ to estimate the directions of vectors $\mathbf{p}_j$.
\section{Numerical study}
\label{se:results}
 {In this section, we present numerical examples to determine three-dimensional electromagnetic sources using our proposed method. The simulation was done using the computing software MATLAB. 
Throughout our examples, the sampling domain is the cube $[-1.5, 1.5]^3$, which contains the unknown sources and is uniformly discretized into $200$ sampling points in each direction, resulting in a step size of approximately $0.015$. In most of our examples, we consider  $k = 20$ and   $s=4$ in the imaging functions. Sections \ref{part:k} and \ref{part:s} include the numerical examples for different values of $k$ and $s$, respectively. 
}The synthetic Cauchy data $\E(\x)$ and $\curl \E(\x)\times \bm{\nu}$ 
 are generated by numerical computation  of $\E$ in \eqref{reprep} and  $$\curl \E(\x)  =\sum_{j=1}^N \curl (\Phi(\x,\y)\p_j).$$ 
In most examples, the Cauchy data are measured on a sphere of radius $25$, which is approximately 
$80$ wavelengths, centered at the origin. The case of near field data is presented in section~\ref{nearfield}. These data points are expressed in spherical coordinates using angles $\phi \in [0,\pi]$ and $ \theta \in [0,2\pi]$, each having $100$ uniformly distributed values across their respective ranges. 

To simulate noisy data, we introduce random noise to our synthetic data. More precisely, we incorporate two complex-valued noise vectors $\mathcal{N}_{1,2} \in \mathbb{C}^3$ into the data vectors, each containing numbers $a+ib$ where $a, b\in (-1,1)$ randomly generated from a uniform distribution.  {For simplicity, we consider the same noise level $\delta=10\%$ for $\E$ and $\curl \E\times \bm{\nu}$ in most numerical examples. The performance of the method for highly noisy data is presented in section \ref{part: noise}.} The noisy data are given by 
\begin{align*}
    \E_\delta : =  \E+\delta \frac{\mathcal{N}_1}{\|\mathcal{N}_1\|_2}\|\E\|_2,\quad \curl \E_\delta \times \bm{\nu} : =  \curl \E \times \bm{\nu} +\delta \frac{\mathcal{N}_2}{\|\mathcal{N}_2\|_2}\|\curl \E \times \bm{\nu}\|_2,
\end{align*}
where $\|\cdot \|_2$ is the Euclidean norm. The isovalue used in the 3D isosurface plotting for all examples is
$0.2$ ($20\%$ of the maximal value of the imaging functions that are normalized). All estimated locations and moment vectors are rounded to three decimal digits.


\subsection{Reconstruction results for point sources}
\label{part:point}
We begin with an example involving three   point sources that have moment vectors with comparable magnitudes. The proposed algorithm, which utilizes the imaging functions $\widetilde{I}^\text{re}_s(\mathbf{z})$ and $\widetilde{I}^\text{im}_s(\mathbf{z})$, accurately and quickly determine the number of sources and source locations.
Despite the presence of $10\%$ random noise in the data, the computed moment vectors achieve a relative error of less than $10\%$, demonstrating the robustness of the method. The reconstruction results are presented in Figure \ref{3points} and  Table \ref{Ta:3points}. 
\begin{center}
{\scalebox{0.9}{
    \begin{tabular}{|c|c|c|c|}
    \hline
 \rowcolor{lightgray}    \text{True location} $(N=3)$&  \text{Computed location} & \text{True moment vector} & \text{Computed moment vector} \\
        \hline
     $(-0.9,0,1)$ &  $(-0.900, 0.000,1.005)$ & $\begin{pmatrix}
         -2.5\\4\\-3
     \end{pmatrix}$  &  $\begin{pmatrix}
       -2.471   \\ 3.775  \\ -3.046 
     \end{pmatrix}$  
        \\
     \hline 
     $(-1,0.75,-1)$&$(-1.005, 0.750,-1.005)$& $\begin{pmatrix}
         -1+3i\\5+4i\\3
     \end{pmatrix}$  
     & $\begin{pmatrix}
           -0.924 + 3.000i \\  4.901 + 3.988i\\   2.968 + 0.001i
     \end{pmatrix}$
         \\
         \hline
        $(1.1,-0.3,-1)$&$(1.095, -0.300,-1.005)$ &$\begin{pmatrix}
            4.5i\\-5\\3-2i
        \end{pmatrix}$
    &$ \begin{pmatrix}
        -0.009 + 4.504i\\  -4.922 + 0.031i \\  2.987 - 1.992i
    \end{pmatrix}$
    \\
    \hline
     
    \end{tabular}}
    \captionof{table}{Reconstruction results for three point sources where the moment vectors have comparable magnitudes  {for $k=20, s=4$}. \label{Ta:3points}}}
    \end{center}  
    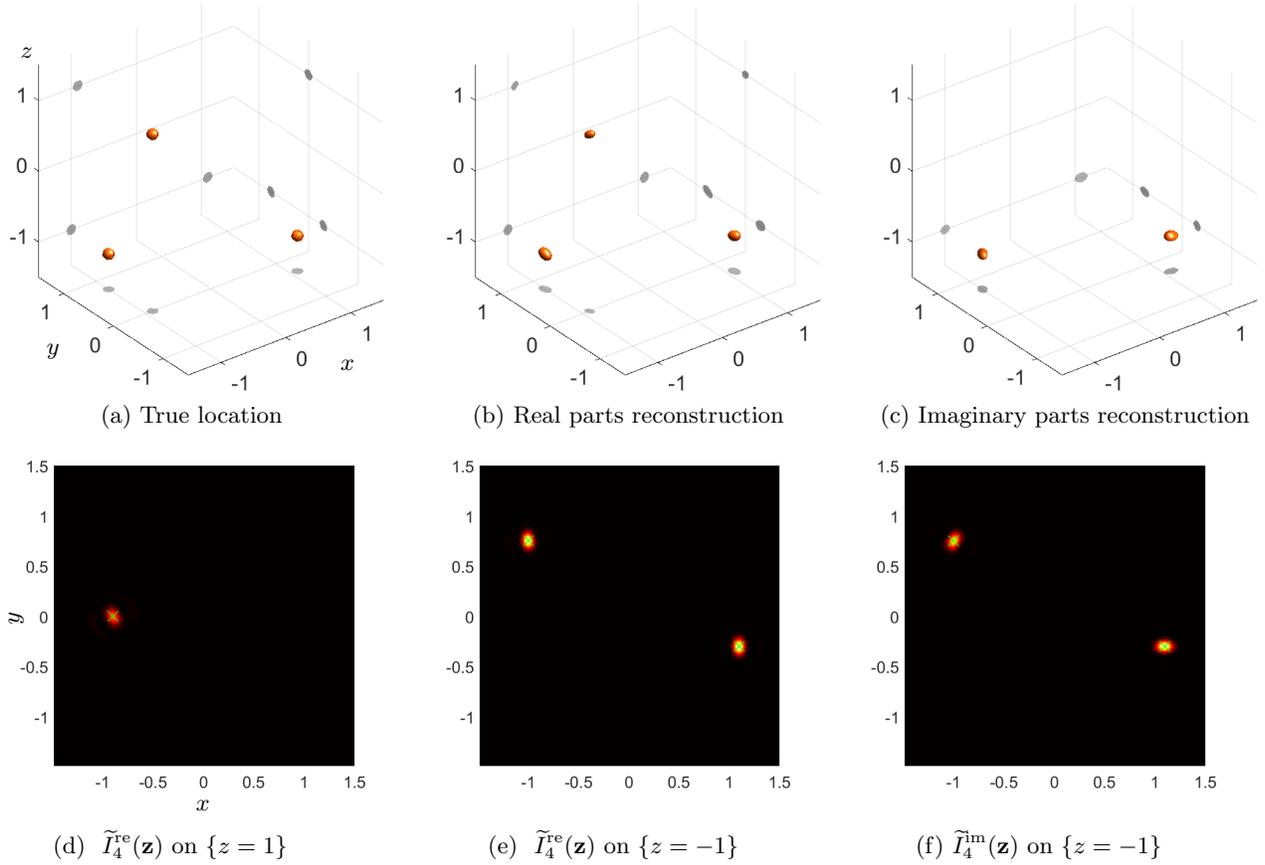
\begin{figure}[h]
    \centering
     \subfloat[True location]{
    \begin{tikzpicture}
    \node[anchor=south west,inner sep=0] (image) at (0,0) {\includegraphics[width=5cm]{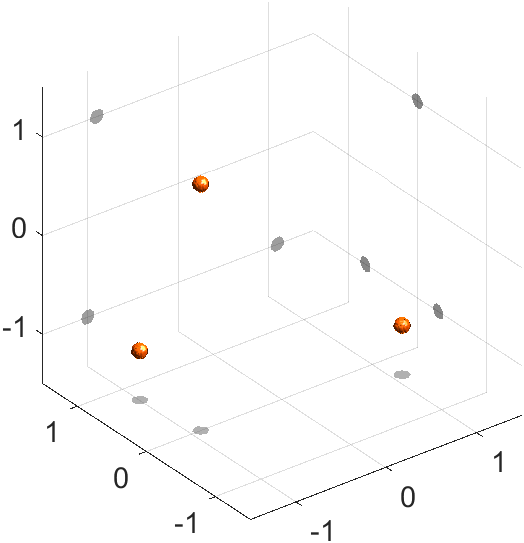}};
    \begin{scope}[x={(image.south east)},y={(image.north west)}]
      \node[anchor=north] at (.9,0.11) {{\footnotesize{$x$}}};
      \node[anchor=south] at (0.12,.06) {\footnotesize{$y$}};
      \node[anchor=south] at (0.05,0.83) {\footnotesize{$z$}};      
    \end{scope}
  \end{tikzpicture}}  
     \hspace{0.55cm}    
     \subfloat[Real parts reconstruction]{
    \begin{tikzpicture}
    \node[anchor=south west,inner sep=0] (image) at (0,0) {\includegraphics[width=5.cm]{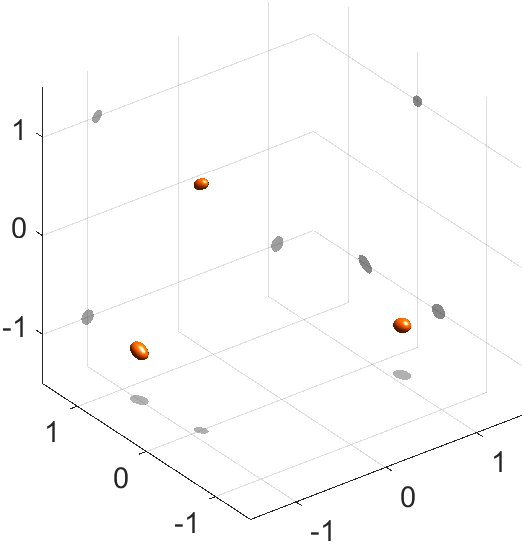}};
    \begin{scope}[x={(image.south east)},y={(image.north west)}]
      \node[anchor=north] at (.9,0.11) {\phantom{\footnotesize{$x$}}};
      \node[anchor=south] at (0.12,.06) {\phantom{\footnotesize{$y$}}};
      \node[anchor=south] at (0.05,0.83) {\phantom{\footnotesize{$x$}}};      
    \end{scope}
  \end{tikzpicture}}  
  \hspace{0.55cm}
     \subfloat[Imaginary parts reconstruction]{
    \begin{tikzpicture}
    \node[anchor=south west,inner sep=0] (image) at (0,0) {\includegraphics[width=5cm]{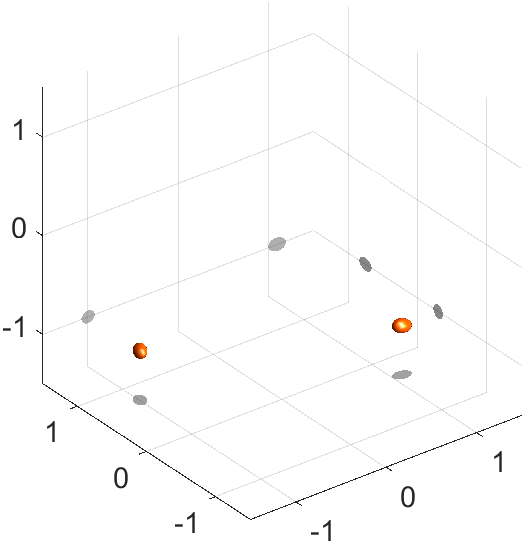}};
    \begin{scope}[x={(image.south east)},y={(image.north west)}]
      \node[anchor=north] at (.9,0.11) {\phantom{\footnotesize{$x$}}};
      \node[anchor=south] at (0.12,.06) {\phantom{\footnotesize{$x$}}};
      \node[anchor=south] at (0.05,0.83) {\phantom{\footnotesize{$x$}}};      
    \end{scope}
  \end{tikzpicture}}  
     \\
    \subfloat[ $\widetilde {I}^\text{re}_4(\z)  \text{ on } \{z = 1\} $]{
  \begin{tikzpicture}
    \node[anchor=south west,inner sep=0] (image) at (0,0) {\includegraphics[width=5.2cm]{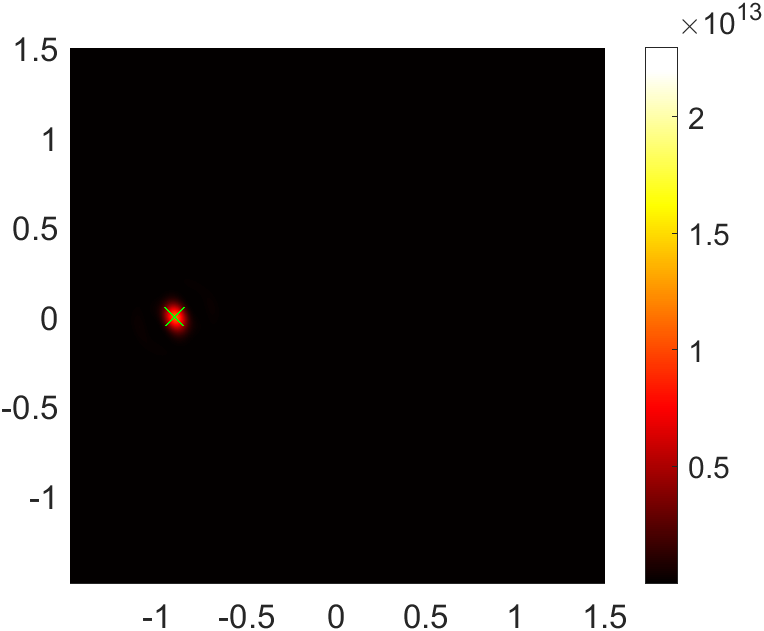}};
    \begin{scope}[x={(image.south east)},y={(image.north west)}]
      \node[anchor=north] at (.53,0) {\footnotesize{$x$}};
      \hspace{0.2cm}
      \node[anchor=south,rotate=90] at (0,.52) {\footnotesize{$y$}};
    \end{scope}
  \end{tikzpicture}}
  \hspace{0.3cm}
  \subfloat[ $\widetilde {I}^\text{re}_4(\z)  \text{ on } \{z = -1\} $]{
  \begin{tikzpicture}
    \node[anchor=south west,inner sep=0] (image) at (0,0) {\includegraphics[width=5.3cm]{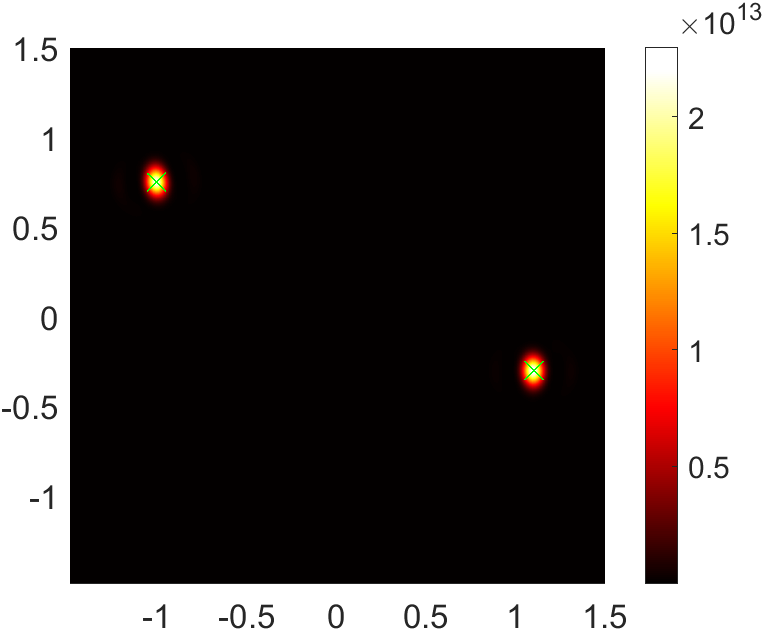}};
    \begin{scope}[x={(image.south east)},y={(image.north west)}]
      \node[anchor=north] at (.5,0) {\phantom{\footnotesize{$x$}}};
    \end{scope}
  \end{tikzpicture}}
   \hspace{0.3cm}
  \subfloat[$\widetilde {I}^\text{im}_4(\z) \text{ on } \{z = -1\} $]{
  \begin{tikzpicture}
    \node[anchor=south west,inner sep=0] (image) at (0,0) {\includegraphics[width=5.2cm]{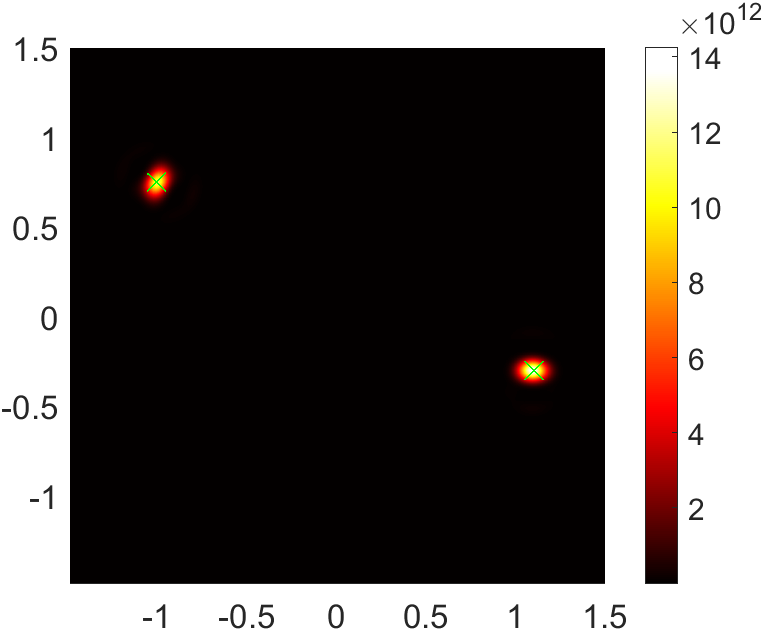}};
    \begin{scope}[x={(image.south east)},y={(image.north west)}]
      \node[anchor=north] at (.5,0) {\phantom{\footnotesize{$x$}}};
    \end{scope}
  \end{tikzpicture}}
    \caption{
    Reconstruction results for the three point sources in Table \ref{Ta:3points}  {for $k=20, s=4$}. Isosurface visualizations for the true locations in (a), for $\widetilde{I}^\text{re}_4(\z)$ in (b), and for $\widetilde{I}^\text{im}_4(\z)$ in (c). Cross-sectional views restricted to 2D domains of $\widetilde{I}^\text{re}_4(\z)$ in (d)--(e) and of $\widetilde{I}^\text{im}_4(\z)$ in (f), with the true locations marked with green crosses.
    }
    \label{3points}
\end{figure}

Our next example considers the case of six point sources where some moment vectors can have notably different magnitudes. 
 As shown in Table \ref{Ta:6points} and Figures \ref{Fi:6points_3d}--\ref{Fi:6points_2d_im}, the two  imaging functions  $\widetilde {I}^\text{re}_s(\z)$ and $ \widetilde {I}^\text{im}_s(\z)$, along with their updated versions based on the proposed numerical algorithm, efficiently determine the number of sources and their locations with high accuracy. The computed moment vectors further demonstrate the effectiveness and robustness of the method, achieving a low relative error of less than $10$\% from the noisy Cauchy data.

\begin{center}
{\scalebox{0.9}{
    \begin{tabular}{|c|c|c|c|}
    \hline
 \rowcolor{lightgray}    \text{True location} $(N=6)$&  \text{Computed location} & \text{True moment vector} & \text{Computed moment vector} \\
        \hline
     $(-1.2,0,-1)$ &  $(-1.200, 0.000,-1.005)$ & $\begin{pmatrix}
        80+11i \\  50+16i \\  -32i
     \end{pmatrix}$  &  $\begin{pmatrix}
       79.776 +10.977i\\
    50.359+15.938i\\
          -0.308-31.462i \\
     \end{pmatrix}$  
        \\
     \hline 
     $(0.6,-1,-1)$&$(0.600, -1.005,-1.005)$& $\begin{pmatrix}
          12-23i\\  35 \\ 3+60i
     \end{pmatrix}$  
     & $\begin{pmatrix}
        11.536-22.805i\\
  34.779-0.709i\\
   2.705+60.789i
     \end{pmatrix}$
         \\
         \hline
        $(1,0.5,0)$&$(0.990, 0.495,0.000)$ &$\begin{pmatrix}
             -6\\  7+40i\\ -18+5i
        \end{pmatrix}$
    &$ \begin{pmatrix}
        -5.923-0.187i\\
      6.998+39.645i\\
-18.781+5.360i
    \end{pmatrix}$
    \\
    \hline
        $(-0.3,0,0)$&$(-0.300, 0.000,0.000)$ &$\begin{pmatrix}
             -5i\\  12\\ 9+14i
        \end{pmatrix}$
    &$ \begin{pmatrix}
               -0.049-5.013i\\
  12.209+0.093i\\
     8.416+14.025i

    \end{pmatrix}$\\
      \hline
       $(-1,0.8,1)$&$(-1.005, 0.795,1.005)$ &$\begin{pmatrix}
             7-26i\\  -2\\ 8 
        \end{pmatrix}$
    &$ \begin{pmatrix}          
    7.049-25.951i\\
    -1.977+0.308i\\
    8.015-0.257i 
    \end{pmatrix}$\\
      \hline
       $(0,-1,1)$&$(0.000, -1.005,1.005)$ &$\begin{pmatrix}
             25\\10\\6
        \end{pmatrix}$
    &$ \begin{pmatrix}          
    24.927  \\  9.379  \\  5.160
    \end{pmatrix}$\\
      \hline
    \end{tabular}}
    \captionof{table}{ Reconstruction results for six point sources for which some moment vectors have notably different  magnitudes  {for $k=20, s=4$}.   
    \label{Ta:6points}}}
    \end{center}

\begin{figure}[h]
    \centering
    \subfloat[True location]{
    \begin{tikzpicture}
    \node[anchor=south west,inner sep=0] (image) at (0,0) {\includegraphics[width=5cm]{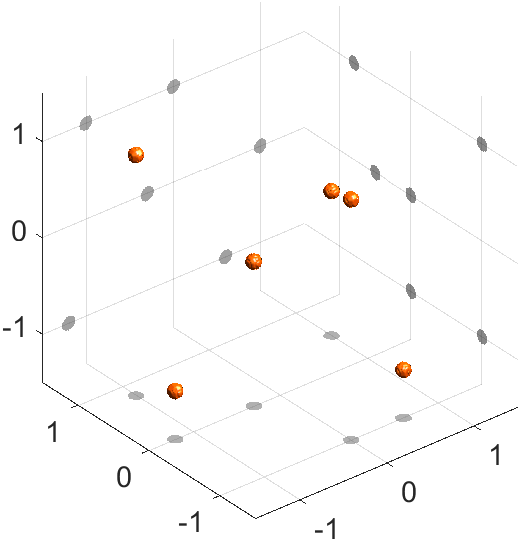}};
    \begin{scope}[x={(image.south east)},y={(image.north west)}]
      \node[anchor=north] at (.9,0.11) {{\footnotesize{$x$}}};
      \node[anchor=south] at (0.12,.06) {\footnotesize{$y$}};
      \node[anchor=south] at (0.05,0.83) {\footnotesize{$z$}};      
    \end{scope}
  \end{tikzpicture}} 
     \hspace{0.55cm}
        \subfloat[Real parts reconstruction]{
    \begin{tikzpicture}
    \node[anchor=south west,inner sep=0] (image) at (0,0) {\includegraphics[width=5cm]{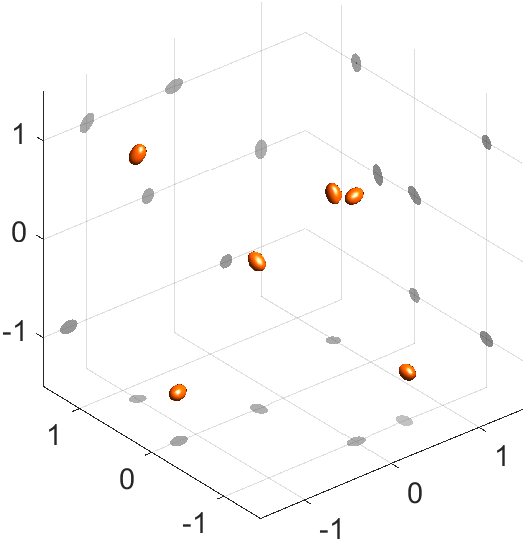}};
    \begin{scope}[x={(image.south east)},y={(image.north west)}]
      \node[anchor=north] at (.9,0.11) {\phantom{\footnotesize{$x$}}};
      \node[anchor=south] at (0.12,.06) {\phantom{\footnotesize{$x$}}};
      \node[anchor=south] at (0.05,0.83) {\phantom{\footnotesize{$x$}}};   \end{scope}
  \end{tikzpicture}} 
      \hspace{0.55cm}
      \subfloat[Imaginary parts reconstruction]{
    \begin{tikzpicture}
    \node[anchor=south west,inner sep=0] (image) at (0,0) {\includegraphics[width=5cm]{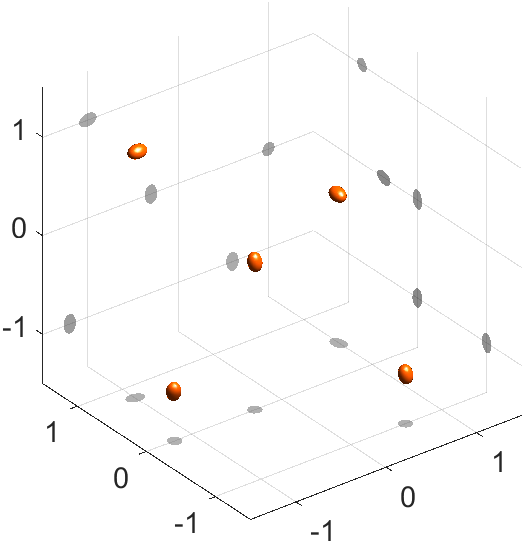}};
    \begin{scope}[x={(image.south east)},y={(image.north west)}]
      \node[anchor=north] at (.9,0.11) {\phantom{\footnotesize{$x$}}};
      \node[anchor=south] at (0.12,.06) {\phantom{\footnotesize{$x$}}};
      \node[anchor=south] at (0.05,0.83) {\phantom{\footnotesize{$x$}}};   \end{scope}
  \end{tikzpicture}} 
    \caption{Reconstruction results for the six point sources in Table \ref{Ta:6points}  {for $k=20, s=4$}. Isosurface visualizations for the true locations in (a), for $\widetilde {I}^\text{re}_4(\z)$ and their updates in (b), and for $\widetilde {I}^\text{im}_4(\z)$ and their updates in (c).}    
    
    \label{Fi:6points_3d}
\end{figure}
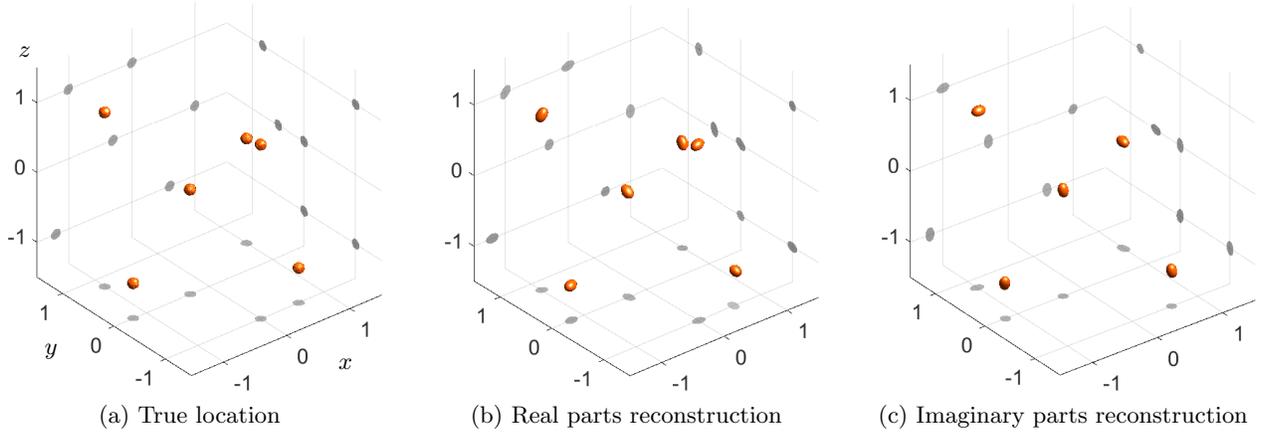 

\begin{figure}[h]
 \centering
 \hspace{-0.4cm}
  \begin{tikzpicture}
    \node[anchor=south west,inner sep=0] (image) at (0,0) {\includegraphics[width=5.2cm]{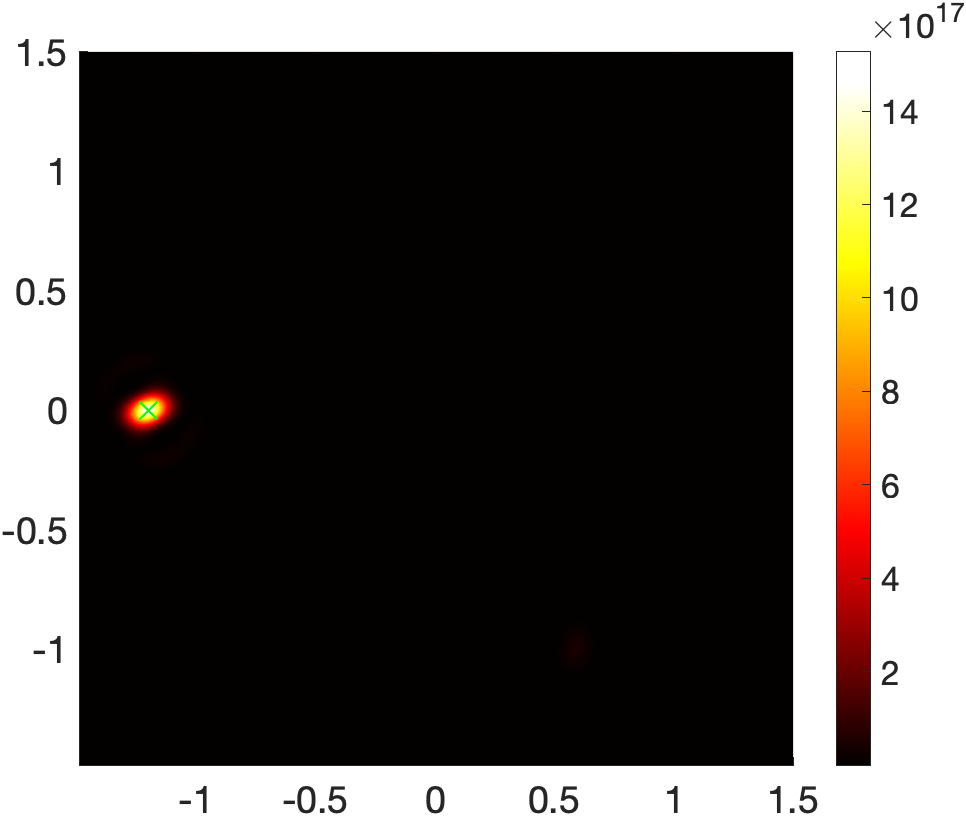}};
    \begin{scope}[x={(image.south east)},y={(image.north west)}]
      \node[anchor=north] at (.53,0) {\footnotesize{$x$}};
      \hspace{0.2cm}
      \node[anchor=south,rotate=90] at (0,.52) {\footnotesize{$y$}};
    \end{scope}
  \end{tikzpicture}   
  \hspace{.3 cm}
    \begin{tikzpicture}
    \node[anchor=south west,inner sep=0] (image) at (0,0) {\includegraphics[width=5.2cm]{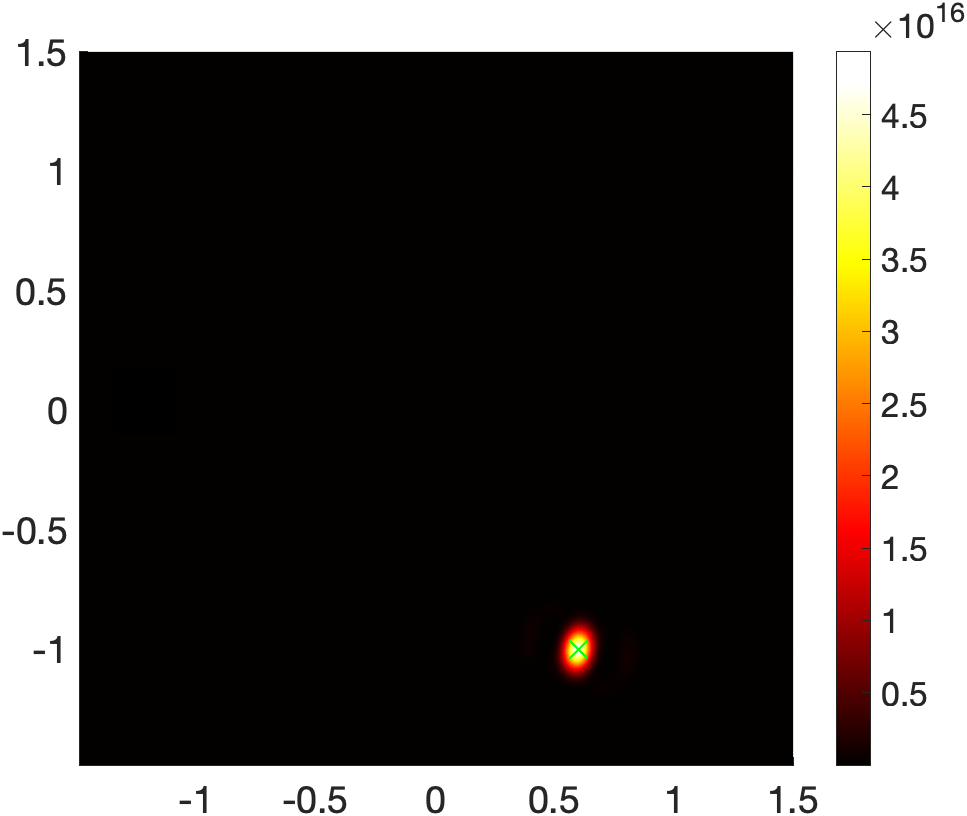}};
    \begin{scope}[x={(image.south east)},y={(image.north west)}]
      \node[anchor=north] at (.53,0) {\phantom{\footnotesize{$x$}}};
    \end{scope}
  \end{tikzpicture}
   \hspace{.3 cm}
    \begin{tikzpicture}
    \node[anchor=south west,inner sep=0] (image) at (0,0) {\includegraphics[width=5.2cm]{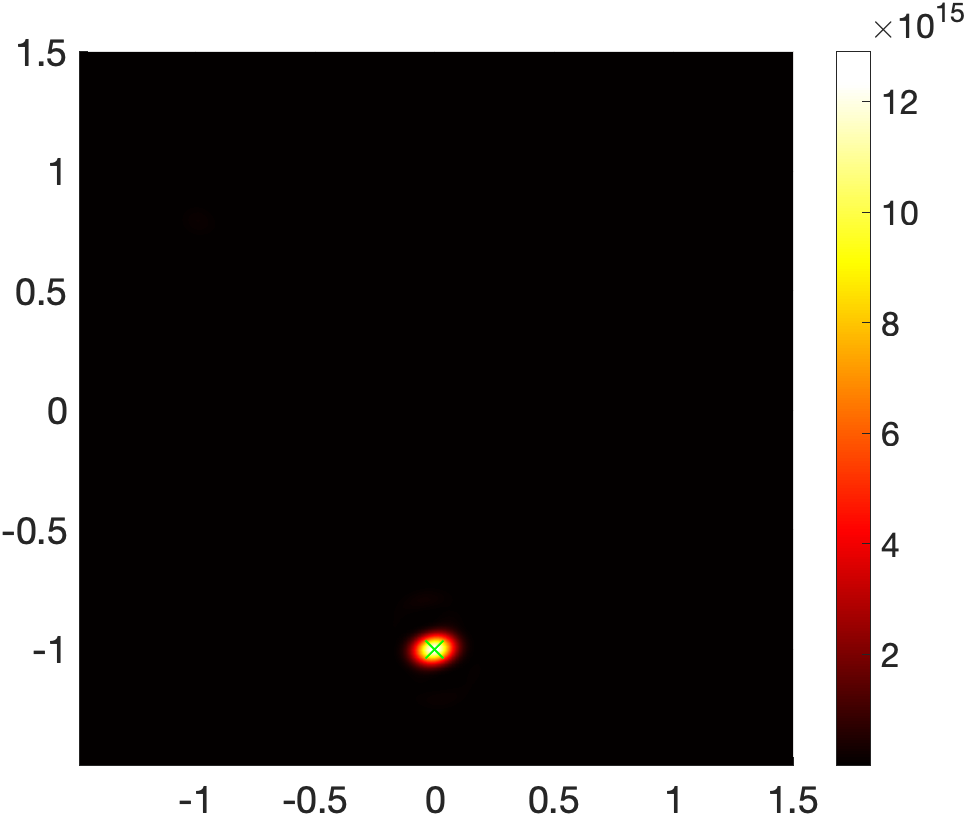}};
    \begin{scope}[x={(image.south east)},y={(image.north west)}]
      \node[anchor=north] at (.53,0) {\phantom{\footnotesize{$x$}}};
    \end{scope}
  \end{tikzpicture}
    \\
    \vspace{0.3cm}
    \hspace{-0.2cm}
  \begin{tikzpicture}
    \node[anchor=south west,inner sep=0] (image) at (0,0) {\includegraphics[width=5.2cm]{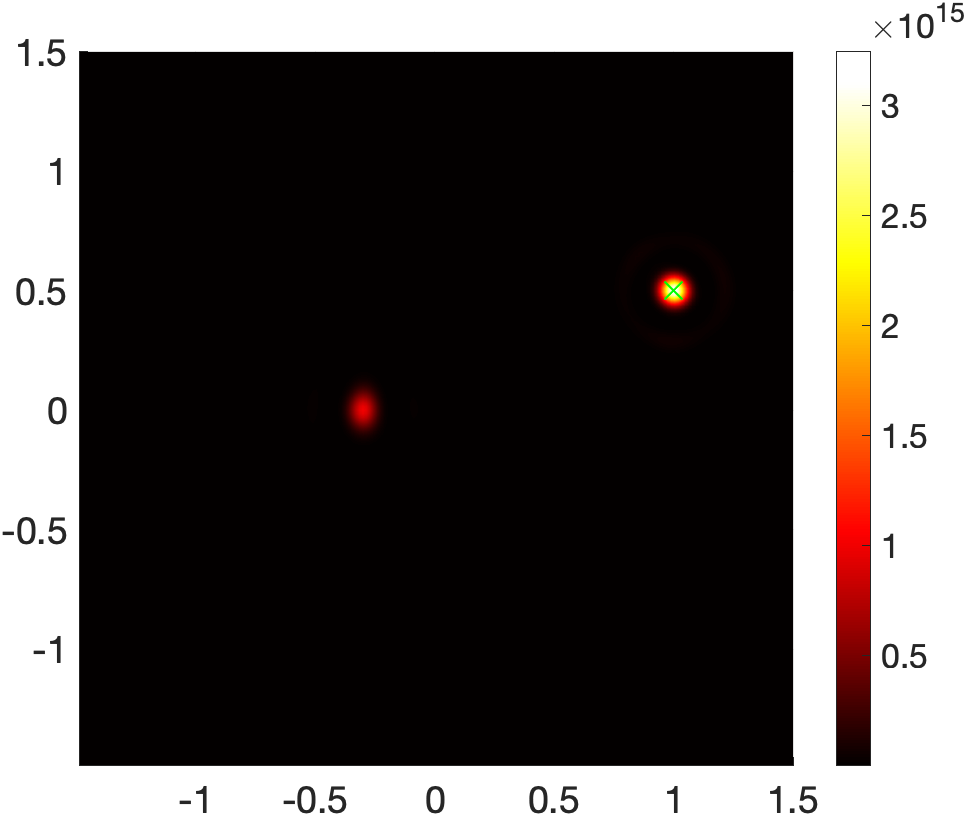}};
    \begin{scope}[x={(image.south east)},y={(image.north west)}]
      \node[anchor=north] at (.53,0) {\phantom{\footnotesize{$x$}}};
      \hspace{0.2cm}
      \node[anchor=north] at (0,0.52) {\phantom{\footnotesize{$x$}}};
    \end{scope}
  \end{tikzpicture}
  \hspace{.3 cm}
  \begin{tikzpicture}
    \node[anchor=south west,inner sep=0] (image) at (0,0) {\includegraphics[width=5.2cm]{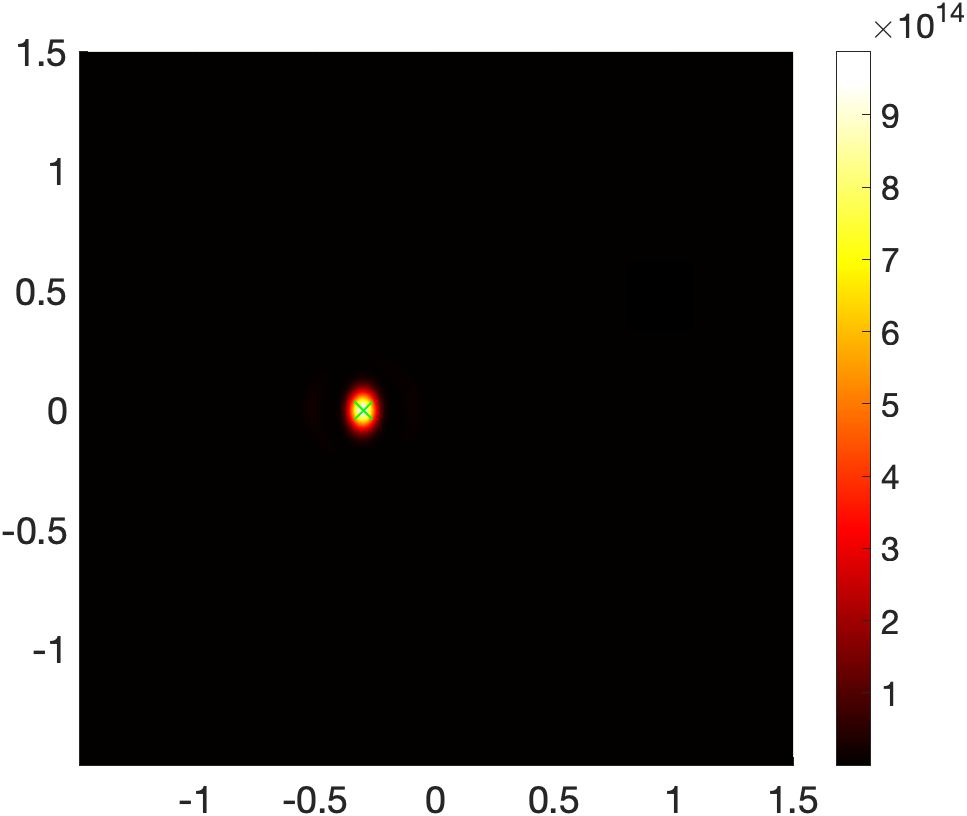}};
    \begin{scope}[x={(image.south east)},y={(image.north west)}]
      \node[anchor=north] at (.53,0) {\phantom{\footnotesize{$x$}}};
    \end{scope}
  \end{tikzpicture}
   \hspace{.3 cm}
  \begin{tikzpicture}
    \node[anchor=south west,inner sep=0] (image) at (0,0) {\includegraphics[width=5.2cm]{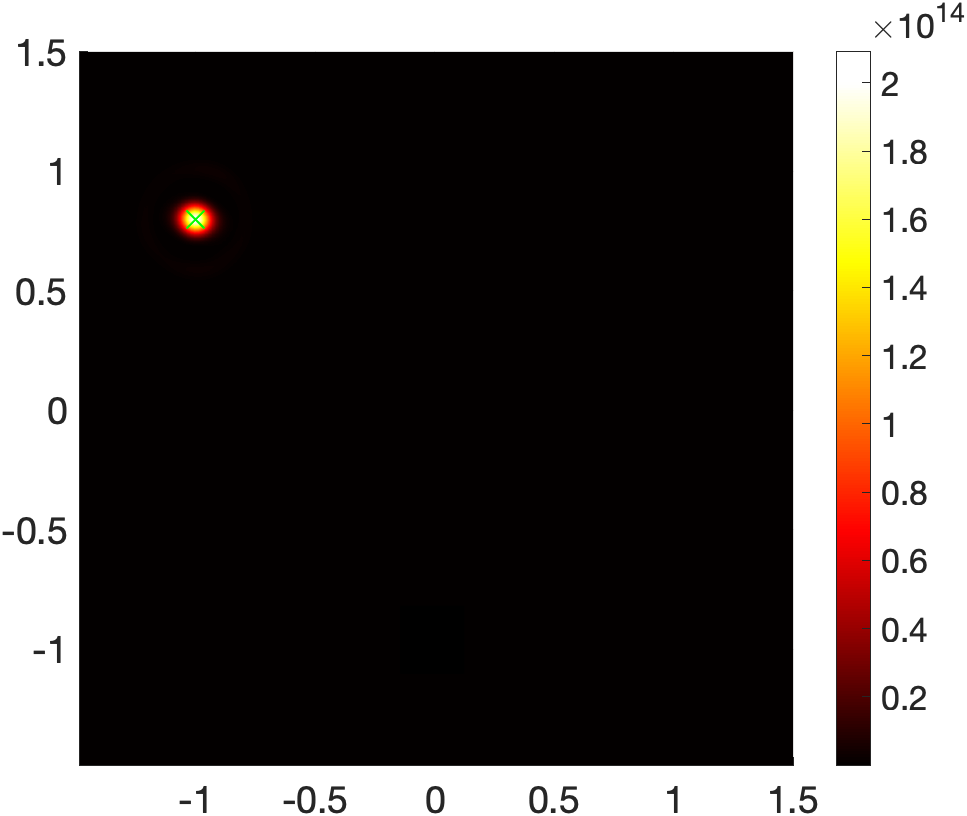}};
    \begin{scope}[x={(image.south east)},y={(image.north west)}]
      \node[anchor=north] at (.53,0) {\phantom{\footnotesize{$x$}}};
    \end{scope}
  \end{tikzpicture}
    \caption{Reconstruction results for the six point sources in Table \ref{Ta:6points}  {for $k=20, s=4$}. Cross-sectional views restricted to 2D domains of $\widetilde{I}^\text{re}_4(\mathbf{z})$ and their updates, with the true locations marked with green crosses.}
    \label{Fi:6points_2d_real}
\end{figure} 
\begin{figure}[ht!]
    \centering
    \hspace{-0.4cm}
  \begin{tikzpicture}
    \node[anchor=south west,inner sep=0] (image) at (0,0) {\includegraphics[width=5.2cm]{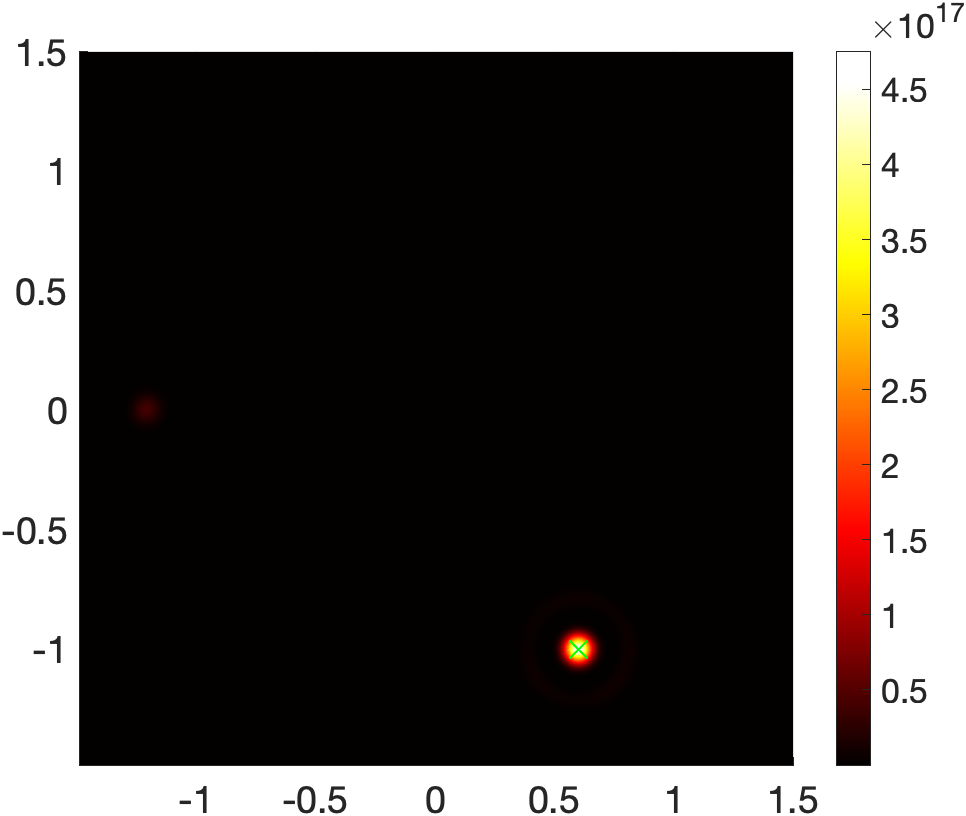}};
    \begin{scope}[x={(image.south east)},y={(image.north west)}]
      \node[anchor=north] at (.53,0) {\footnotesize{$x$}};
      \hspace{0.2cm}
      \node[anchor=south,rotate=90] at (0,.52) {\footnotesize{$y$}};
    \end{scope}
  \end{tikzpicture}   
    \begin{tikzpicture}
    \node[anchor=south west,inner sep=0] (image) at (0,0) {\includegraphics[width=5.2cm]{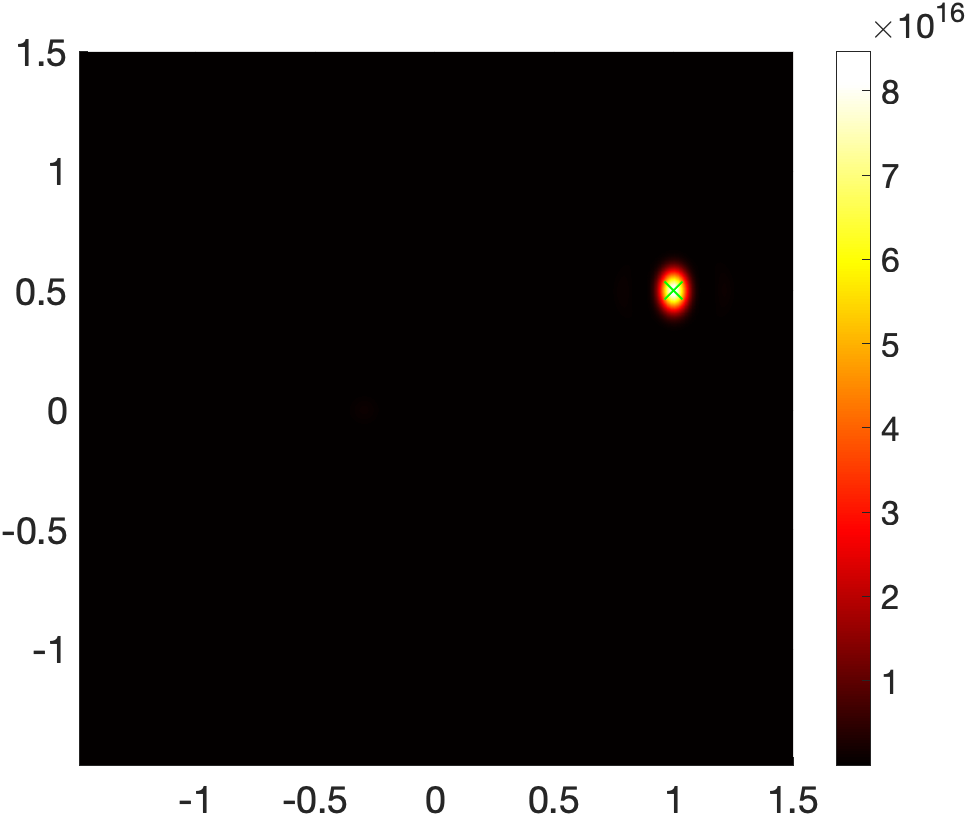}};
    \begin{scope}[x={(image.south east)},y={(image.north west)}]
      \node[anchor=north] at (.53,0) {\phantom{\footnotesize{$x$}}};
      \hspace{0.2cm}
      \node[anchor=south,rotate=90] at (0,.52)  {\phantom{\footnotesize{$x$}}};
    \end{scope}
  \end{tikzpicture}   
   \begin{tikzpicture}
    \node[anchor=south west,inner sep=0] (image) at (0,0) {\includegraphics[width=5.2cm]{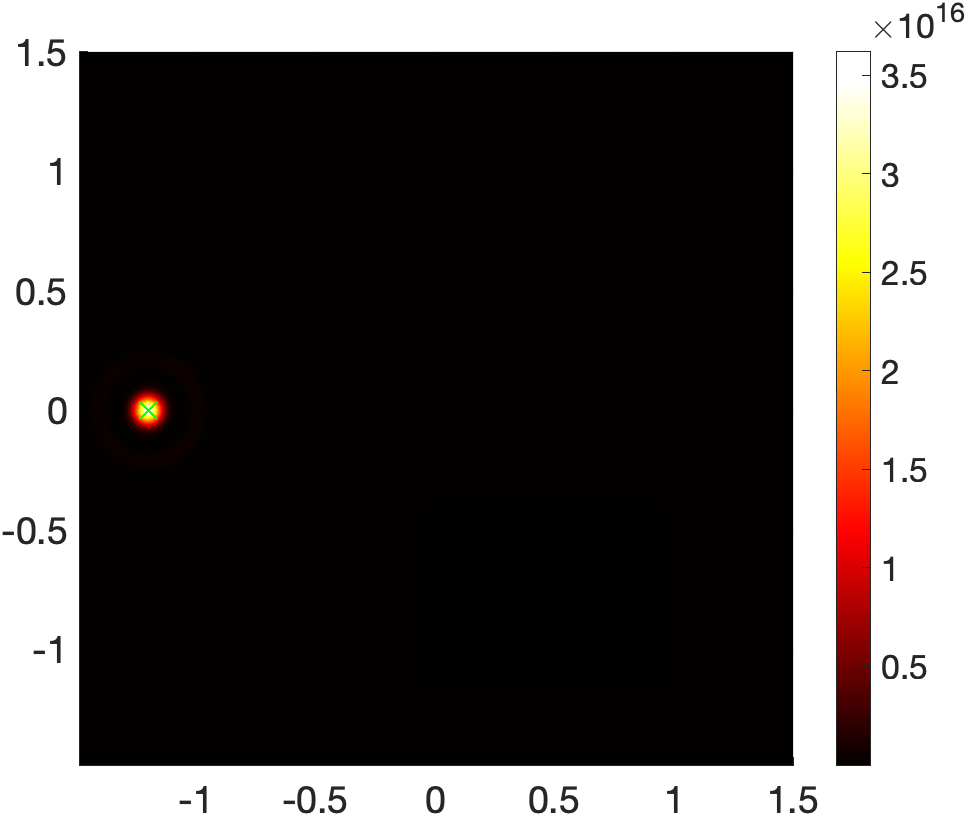}};
    \begin{scope}[x={(image.south east)},y={(image.north west)}]
      \node[anchor=north] at (.53,0) {\phantom{\footnotesize{$x$}}};
      \hspace{0.2cm}
      \node[anchor=south,rotate=90] at (0,.52)  {\phantom{\footnotesize{$x$}}};
    \end{scope}
  \end{tikzpicture}   
  \\ \vspace{0.3cm}
   \begin{tikzpicture}
    \node[anchor=south west,inner sep=0] (image) at (0,0) {\includegraphics[width=5.2cm]{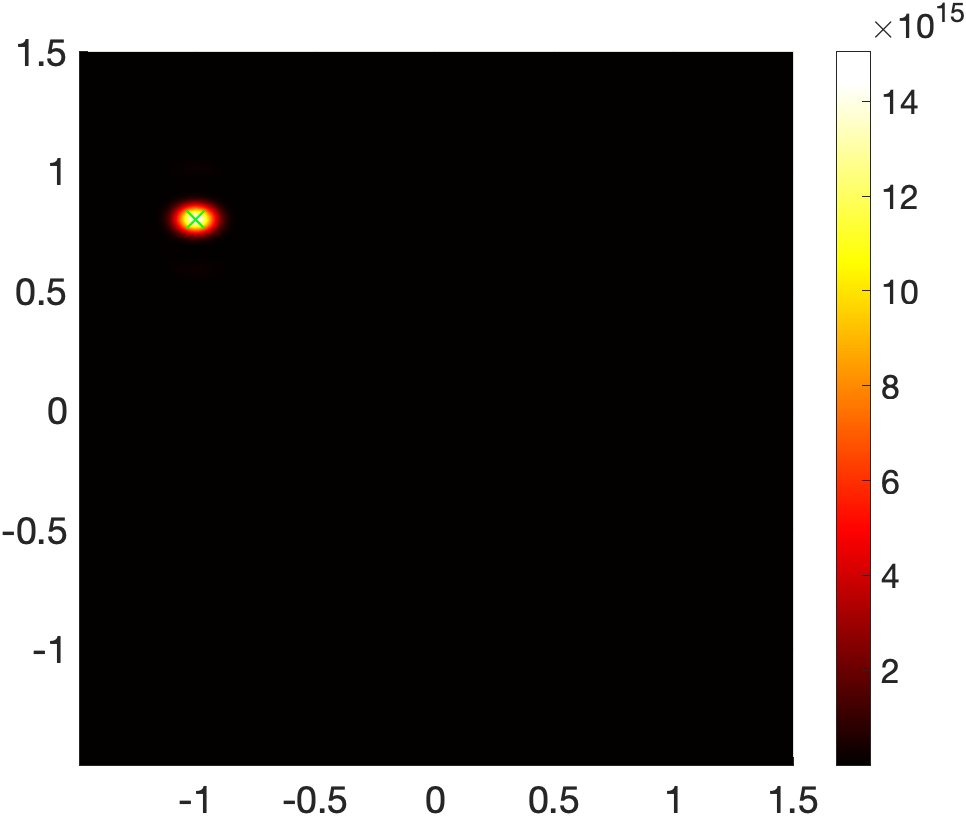}};
    \begin{scope}[x={(image.south east)},y={(image.north west)}]
      \node[anchor=north] at (.53,0) {\phantom{\footnotesize{$x$}}};
      \hspace{0.2cm}
      \node[anchor=south,rotate=90] at (0,.52)  {\phantom{\footnotesize{$x$}}};
    \end{scope}
  \end{tikzpicture}   
   \begin{tikzpicture}
    \node[anchor=south west,inner sep=0] (image) at (0,0) {\includegraphics[width=5.2cm]{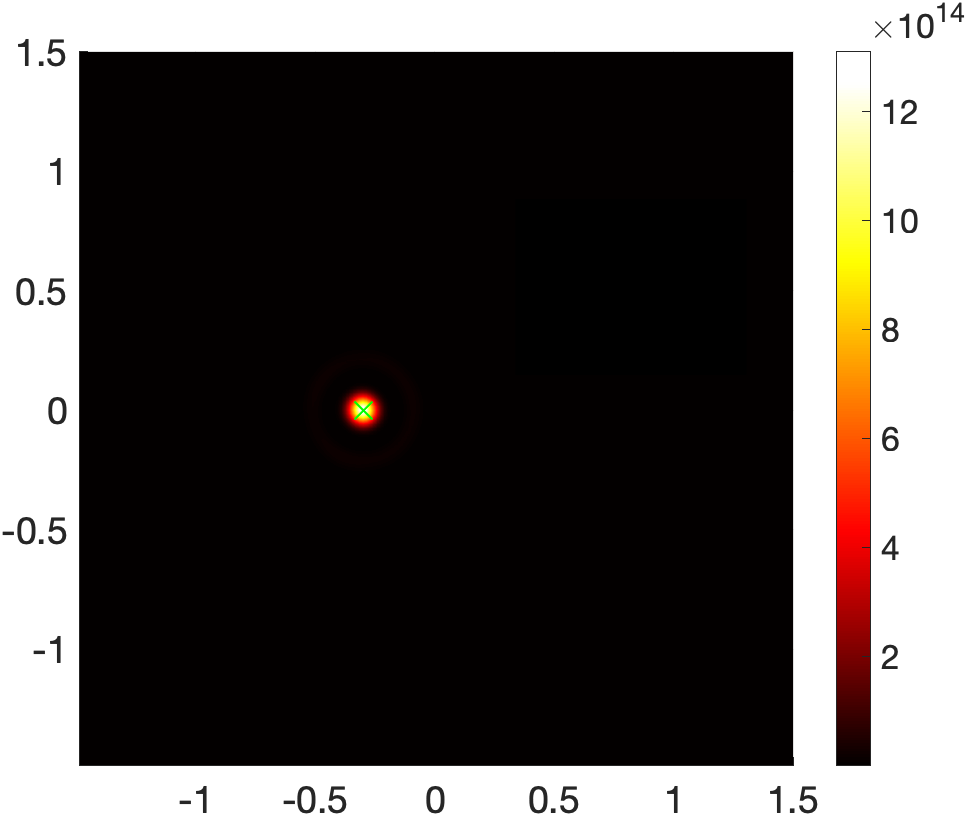}};
    \begin{scope}[x={(image.south east)},y={(image.north west)}]
      \node[anchor=north] at (.53,0) {\phantom{\footnotesize{$x$}}};
      \hspace{0.2cm}
      \node[anchor=south,rotate=90] at (0,.52)  {\phantom{\footnotesize{$x$}}};
    \end{scope}
  \end{tikzpicture}   
    \caption{Reconstruction results for the six point sources in Table \ref{Ta:6points}  {for $k=20, s=4$}. Cross-sectional views restricted to 2D domains of  $\widetilde {I}^\text{im}_4(\z)$ and their updates, with the true locations marked with green crosses.}
    \label{Fi:6points_2d_im}
\end{figure} 

\subsection {Reconstruction for small-volume sources}
\label{part:small}
In this part, we investigate our method to identify points within small-volume sources and determine the directions of constant vectors characterizing the sources. We examine six electromagnetic ball sources with small radii, each having a possibly complex vector with magnitudes that are not comparable. Despite the presence of $10\%$ random noise in the measurement data, the method quickly and accurately locates the centers of these source supports. Furthermore, the estimated directions of the constant vectors exhibit a low relative error, remaining below $10$\%. See further details of the results in Table \ref{Ta:6balls}.

\begin{center}
{\scalebox{0.9}{
    \begin{tabular}{|c|c|c|c|c|}
    \hline
 \rowcolor{lightgray}   \text{True $\x_j$} ($N=6$)&  \text{Computed  $\x_j$} & \text{Radius}  &\text{True $\p_j$} &\text{Computed $\p_j/|\p_j|$ }  \\ 
        \hline
        $(1,0,1.2)$ &$ (1.005, 0.000,1.200)$ &$0.11$ &
      $ 85.170\begin{pmatrix}
         0.317 + 0.234i\\  -0.821\\  - 0.410i
       \end{pmatrix}$
       &  $\begin{pmatrix}
         0.313 + 0.239i \\ -0.819 - 0.002i  \\-0.004 - 0.416i
    \end{pmatrix}$
       \\
          \hline
    $(-1,-0.6,1.2)
    $&$(-1.005, -0.600,1.200) $&  $0.12$ & $57.524 \begin{pmatrix}
       0.574 - 0.244i\\  -0.173 + 0.695i\\   0.312
    \end{pmatrix}$
   
    & $\begin{pmatrix}
          0.577 - 0.246i \\ -0.173 + 0.692i\\   0.310 + 0.006i
    \end{pmatrix}$
        \\
     \hline 
    $(-1,0,-1)$ &  $ (-1.005, 0.000,-0.990)$ &$0.11$ &$26.571 \begin{pmatrix}
      0.565\\   -0.338\\-0.753
    \end{pmatrix}$
   
    & $\begin{pmatrix}
           0.568  \\ -0.330 \\  -0.753
    \end{pmatrix}$
         \\
         \hline
       $(1,0.3,-1)$ &$(1.005, 0.300,-1.005)$ & $0.13$ &$29.547\begin{pmatrix}
       0.334i\\  - 0.575i\\ 0.745i
    \end{pmatrix}$
   
    & $\begin{pmatrix}
         0.344i \\  -0.574i \\  0.744i
    \end{pmatrix}$
    \\
    \hline
    $(1.1,-0.7,0)$ &$(1.095, -0.705,0.000)$ & $0.1$ & $25.593 \begin{pmatrix}
     0.508 + 0.351i\\  -0.468 + 0.586i\\   - 0.234i
    \end{pmatrix}$
   
    & $\begin{pmatrix}
        0.506 + 0.347i \\ -0.468 + 0.590i   \\0.0000 - 0.234i
    \end{pmatrix}$
    \\
    \hline
        $(0,0.5,0)$ &$ (0.000, 0.495,0.000)$& $0.11$ & $ 19.712\begin{pmatrix}
       -0.482 + 0.386i\\   - 0.579i\\   0.531
    \end{pmatrix}$
   
    & $\begin{pmatrix}
           -0.479 + 0.410i \\  0.000 - 0.574i \\  0.522 + 0.0000i
    \end{pmatrix}$
    \\
    \hline
    \end{tabular}}
    \captionof{table}{Reconstruction results for six small-volume ball sources  {for $k=20, s=4$}, each characterized by a (possibly complex) constant vector.   
    \label{Ta:6balls}}}
    \end{center}
\subsection{Reconstruction with different wavenumbers}
\label{part:k}
 {In this part, we present numerical examples for reconstructing point sources with different wavenumbers. The data includes $10\%$ of random noise and the power $s=4$ is maintained in the imaging functions.}  {We first consider three point sources with the same true locations and true moment vectors as in Table \ref{Ta:3points}.
Reconstruction results for $k=8$ are shown in Table \ref{Ta:3points_k8} and Figure \ref{3points_k8}.}  {Comparing these with 
the results for  $k=20$   in Table \ref{Ta:3points}, the relative errors in the computed locations
range from $1.093\%-1.175\%$ for $k=8$, whereas they are slightly lower, ranging from $0.372\%-0.466\%$ for $k=20$. Moreover, the relative errors in the computed moment vectors for $k=20$
range from $1.150\%-4.140\%$, while for $k=8$, they increase to $3.811\%-8.921\%$ but remain within a reasonable range under noisy data. For $k=30$, the results slightly improve compared to those for $k=20$.
}

\begin{center}
{\scalebox{0.9}{
    \begin{tabular}{|c|c|c|c|}
    \hline
 \rowcolor{lightgray}    \text{True location} &  \text{Computed location}& \text{True moment} & \text{Computed moment}   \\
        \hline
     $(-0.9,0,1)$ &  $(-0.915, 0.000,1.005) $ & $\begin{pmatrix}
         -2.5\\4\\-3
     \end{pmatrix}$  &  $\begin{pmatrix}
      -2.404\\   3.543\\   -3.175 

     \end{pmatrix}$  
        \\
     \hline 
     $(-1,0.75,-1)$&$(-0.990, 0.735,-1.005)$ & $\begin{pmatrix}
         -1+3i\\5+4i\\3
     \end{pmatrix}$  
     & $\begin{pmatrix}
            -0.817+3.340i     \\   4.861+3.932i \\   2.838-0.011i
     \end{pmatrix}$
         \\
         \hline
        $(1.1,-0.3,-1)$&$(1.095, -0.285,-1.005)$ 
        &$\begin{pmatrix}
            4.5i\\-5\\3-2i
        \end{pmatrix}$
    &$ \begin{pmatrix}
      0.129+4.462i \\  -5.230-0.036i  \\  3.111-1.996i  
    \end{pmatrix}$ 
    \\
    \hline     
    \end{tabular}}
    \captionof{table}{ {Reconstruction results for three point sources for $k=8, s=4$. }\label{Ta:3points_k8}}} 
    \end{center} 

 \begin{figure}[h]
    \centering
     \subfloat[True location]{
    \begin{tikzpicture}
    \node[anchor=south west,inner sep=0] (image) at (0,0) {\includegraphics[width=5cm]{3points3dtrue1.png}};
    \begin{scope}[x={(image.south east)},y={(image.north west)}]
      \node[anchor=north] at (.9,0.11) {{\footnotesize{$x$}}};
      \node[anchor=south] at (0.12,.06) {\footnotesize{$y$}};
      \node[anchor=south] at (0.05,0.83) {\footnotesize{$z$}};      
    \end{scope}
  \end{tikzpicture}}  
     \hspace{0.55cm}
     \subfloat[Real parts reconstruction]{
    \begin{tikzpicture}
    \node[anchor=south west,inner sep=0] (image) at (0,0) {\includegraphics[width=5cm]{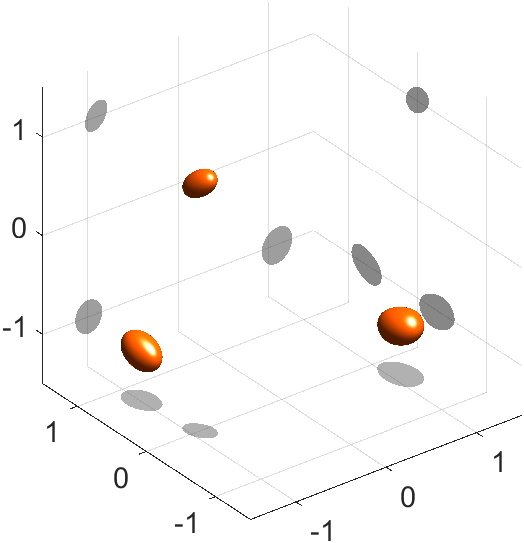}};
    \begin{scope}[x={(image.south east)},y={(image.north west)}]
      \node[anchor=north] at (.9,0.11) {\phantom{\footnotesize{$x$}}};
      \node[anchor=south] at (0.12,.06) {\phantom{\footnotesize{$y$}}};
      \node[anchor=south] at (0.05,0.83) {\phantom{\footnotesize{$x$}}};      
    \end{scope}
  \end{tikzpicture}}  
  \hspace{0.55cm}
     \subfloat[Imaginary parts reconstruction]{
    \begin{tikzpicture}
    \node[anchor=south west,inner sep=0] (image) at (0,0) {\includegraphics[width=5cm]{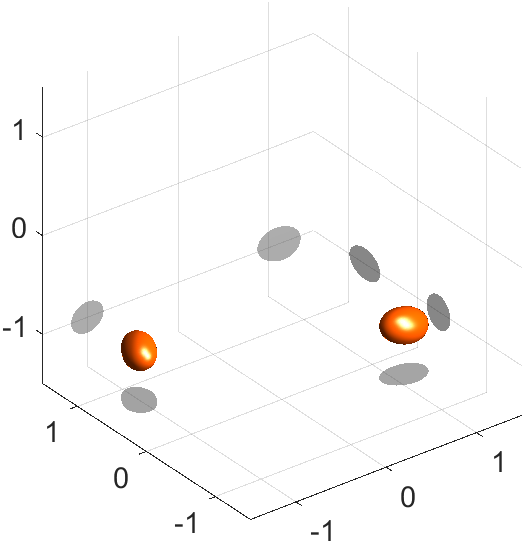}};
    \begin{scope}[x={(image.south east)},y={(image.north west)}]
      \node[anchor=north] at (.9,0.11) {\phantom{\footnotesize{$x$}}};
      \node[anchor=south] at (0.12,.06) {\phantom{\footnotesize{$x$}}};
      \node[anchor=south] at (0.05,0.83) {\phantom{\footnotesize{$x$}}};      
    \end{scope}
  \end{tikzpicture}}  
    \caption{
     {Reconstruction results for the three point sources in Table \ref{Ta:3points_k8} {for $k=8, s=4$}. Isosurface visualizations for the true locations in (a), for $\widetilde{I}^\text{re}_4(\z)$ in (b), and for $\widetilde{I}^\text{im}_4(\z)$ in (c). }
    }
    \label{3points_k8}
\end{figure}

 {Next, we determine six point sources  with the same true locations and true moment vectors as in Table \ref{Ta:6points} at different wavenumbers. Reconstruction results for $k=12$ are shown in Table \ref{Ta:6points_k12}. The computed locations for $k=8$ and $k=12$ are almost similar to those for $k=20$ in Table \ref{Ta:6points}, with low relative errors not exceeding $1\%$. The relative errors in computed moment vectors range between $0.742\%-3.796\%$ for $k=20$, increasing to $2.221\%-9.641\%$ for $k=12$, and further to $4.183\%-11.180\%$ for $k=8$. In addition, for $k=30$, the errors in reconstruction results are slightly lower than those for $k=20$. These results indicate that larger wavenumbers enhance the resolution of the imaging functions, thereby improving the accuracy of source identification. 
}%
    \begin{center}
{\scalebox{0.9}{
    \begin{tabular}{|c|c|c|c|}
    \hline
 \rowcolor{lightgray}    \text{True location} $(N=6)$&  \text{Computed location} & \text{True moment vector} & \text{Computed moment vector} \\
        \hline
     $(-1.2,0,-1)$ &  $(-1.200, 0.000,-1.005)$ & $\begin{pmatrix}
        80+11i \\  50+16i \\  -32i
     \end{pmatrix}$  &  $\begin{pmatrix}
       79.566+10.473i      \\        48.102+15.847i    \\     0.850-32.516i
     \end{pmatrix}$  
        \\
     \hline 
     $(0.6,-1,-1)$&$(0.600, -0.990,-0.990)$& $\begin{pmatrix}
          12-23i\\  35 \\ 3+60i
     \end{pmatrix}$  
     & $\begin{pmatrix}
      11.983-22.953i    \\   35.059-0.650i     \\    3.07+62.092i
     \end{pmatrix}$
         \\
         \hline
        $(1,0.5,0)$&$(0.990, 0.495,0.000)$ &$\begin{pmatrix}
             -6\\  7+40i\\ -18+5i
        \end{pmatrix}$
    &$ \begin{pmatrix}
        -5.614-1.183i        \\       5.716+39.283i \\     -18.209+5.457i
    \end{pmatrix}$
    \\
    \hline
        $(-0.3,0,0)$&$(-0.300, 0.000,0.000)$ &$\begin{pmatrix}
             -5i\\  12\\ 9+14i
        \end{pmatrix}$
    &$ \begin{pmatrix}
              -0.532-3.415i    \\   12.206+0.393i     \\   8.444+14.919i

    \end{pmatrix}$\\
      \hline
       $(-1,0.8,1)$&$(-1.005, 0.795,0.990)$ &$\begin{pmatrix}
             7-26i\\  -2\\ 8 
        \end{pmatrix}$
    &$ \begin{pmatrix}          
   6.244-25.928i   \\    -2.131+0.771i  \\      8.027-0.419i
    \end{pmatrix}$\\
      \hline
       $(0,-1,1)$&$(0.000, -1.005,1.005)$ &$\begin{pmatrix}
            25 \\10\\6

        \end{pmatrix}$
    &$ \begin{pmatrix}          
  24.771     \\   8.743    \\    5.376
    \end{pmatrix}$\\
      \hline
    \end{tabular}}
    \captionof{table}{  {Reconstruction results for six point sources for $k=12, s=4$}.  
    \label{Ta:6points_k12}} }
    \end{center}

\subsection{Reconstruction with different exponent values $s$}
\label{part:s}
%
 {We provide numerical examples with different values of $s$ in the imaging functions to determine three and six point sources. The data includes $10\%$ of random noise and the wavenumber is chosen at $k=20$. 
Using the same base functions $I(z,\mathbf{e}_i)$ with $i=1,2,3$ as examples in section \ref{part:point}, the computed locations and moment vectors for $s=2$ and $s=6$ are identical to those for $s=4$, with relative errors much less than $10\%$.  Therefore, the reconstruction results for $s=2$ and $s=6$ can be found in Table $\ref{Ta:3points}$ for three point sources and Table $\ref{Ta:6points}$ for six point sources. Figures \ref{3points_k20_s2} and \ref{3points_k20_s6}  display the reconstructions for three point sources for $s=2$ and $s=6$, respectively.
It is worth noting that reconstructing locations depends on $s$ in the imaging functions while reconstructing moment vectors does not. 
 For $s=1$, the relative errors in computed locations for three sources reach up to $28\%$. These results demonstrate the accuracy and robustness of our algorithm when $s$ is large enough. 
Moreover, adjusting different $s$ values in our algorithm is easy and computationally cheap.}
\begin{figure}[h]
    \centering
     \subfloat[True location]{
    \begin{tikzpicture}
    \node[anchor=south west,inner sep=0] (image) at (0,0) {\includegraphics[width=5cm]{3points3dtrue1.png}};
    \begin{scope}[x={(image.south east)},y={(image.north west)}]
      \node[anchor=north] at (.9,0.11) {{\footnotesize{$x$}}};
      \node[anchor=south] at (0.12,.06) {\footnotesize{$y$}};
      \node[anchor=south] at (0.05,0.83) {\footnotesize{$z$}};      
    \end{scope}
  \end{tikzpicture}}  
     \hspace{0.55cm}
     \subfloat[Real parts reconstruction]{
    \begin{tikzpicture}
    \node[anchor=south west,inner sep=0] (image) at (0,0) {\includegraphics[width=5cm]{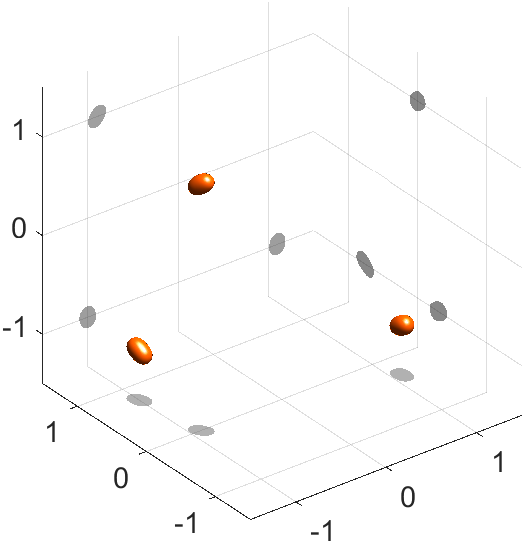}};
    \begin{scope}[x={(image.south east)},y={(image.north west)}]
      \node[anchor=north] at (.9,0.11) {\phantom{\footnotesize{$x$}}};
      \node[anchor=south] at (0.12,.06) {\phantom{\footnotesize{$y$}}};
      \node[anchor=south] at (0.05,0.83) {\phantom{\footnotesize{$x$}}};      
    \end{scope}
  \end{tikzpicture}}  
  \hspace{0.55cm}
     \subfloat[Imaginary parts reconstruction]{
    \begin{tikzpicture}
    \node[anchor=south west,inner sep=0] (image) at (0,0) {\includegraphics[width=5cm]{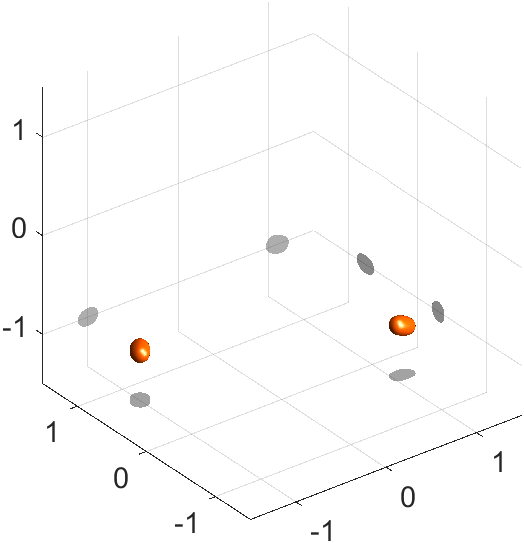}};
    \begin{scope}[x={(image.south east)},y={(image.north west)}]
      \node[anchor=north] at (.9,0.11) {\phantom{\footnotesize{$x$}}};
      \node[anchor=south] at (0.12,.06) {\phantom{\footnotesize{$x$}}};
      \node[anchor=south] at (0.05,0.83) {\phantom{\footnotesize{$x$}}};      
    \end{scope}
  \end{tikzpicture}}  
     \caption{
     {Reconstruction results for the three point sources in Table \ref{Ta:3points} {for $k=20, s=2$}. Isosurface visualizations for the true locations in (a), for $\widetilde{I}^\text{re}_2(\z)$ in (b), and for $\widetilde{I}^\text{im}_2(\z)$ in (c). }
    }
    \label{3points_k20_s2}
\end{figure}

\begin{figure}[h]
    \centering
     \subfloat[True location]{
    \begin{tikzpicture}
    \node[anchor=south west,inner sep=0] (image) at (0,0) {\includegraphics[width=5cm]{3points3dtrue1.png}};
    \begin{scope}[x={(image.south east)},y={(image.north west)}]
      \node[anchor=north] at (.9,0.11) {{\footnotesize{$x$}}};
      \node[anchor=south] at (0.12,.06) {\footnotesize{$y$}};
      \node[anchor=south] at (0.05,0.83) {\footnotesize{$z$}};      
    \end{scope}
  \end{tikzpicture}}  
     \hspace{0.55cm}
     \subfloat[Real parts reconstruction]{
    \begin{tikzpicture}
    \node[anchor=south west,inner sep=0] (image) at (0,0) {\includegraphics[width=5cm]{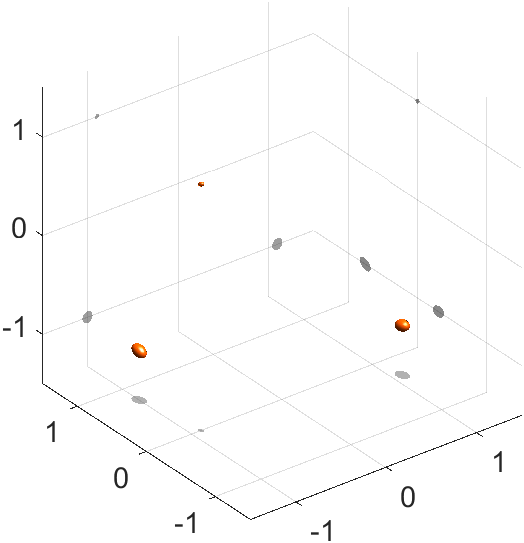}};
    \begin{scope}[x={(image.south east)},y={(image.north west)}]
      \node[anchor=north] at (.9,0.11) {\phantom{\footnotesize{$x$}}};
      \node[anchor=south] at (0.12,.06) {\phantom{\footnotesize{$y$}}};
      \node[anchor=south] at (0.05,0.83) {\phantom{\footnotesize{$x$}}};      
    \end{scope}
  \end{tikzpicture}}  
  \hspace{0.55cm}
     \subfloat[Imaginary parts reconstruction]{
    \begin{tikzpicture}
    \node[anchor=south west,inner sep=0] (image) at (0,0) {\includegraphics[width=5cm]{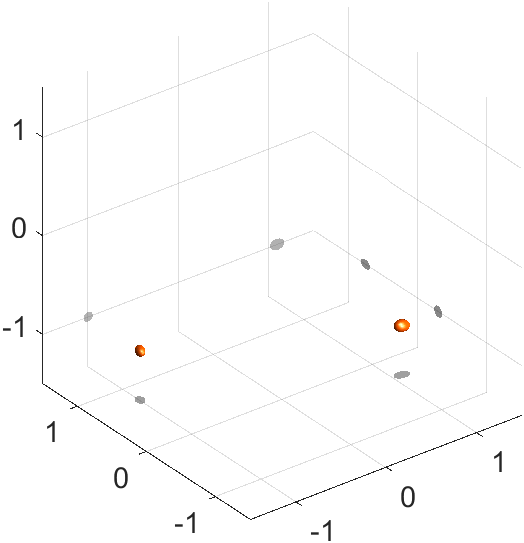}};
    \begin{scope}[x={(image.south east)},y={(image.north west)}]
      \node[anchor=north] at (.9,0.11) {\phantom{\footnotesize{$x$}}};
      \node[anchor=south] at (0.12,.06) {\phantom{\footnotesize{$x$}}};
      \node[anchor=south] at (0.05,0.83) {\phantom{\footnotesize{$x$}}};      
    \end{scope}
  \end{tikzpicture}}  
     \caption{
     {Reconstruction results for the three point sources in Table \ref{Ta:3points} {for $k=20, s=6$}. Isosurface visualizations for the true locations in (a), for $\widetilde{I}^\text{re}_6(\z)$ in (b), and for $\widetilde{I}^\text{im}_6(\z)$ in (c). }
    }
    \label{3points_k20_s6}
\end{figure}

\subsection{Reconstruction with high noise levels} 
 {In this section, we demonstrate the stability of our method by testing its performance under high levels of noise added to synthetic data. We consider the wavenumber $k=20$ and the exponent $s=4$ in the imaging functions.
Reconstruction results for three point sources with \(30\%\) and \(50\%\) noise in the data are presented in Tables \ref{Ta:3points_30noise} and \ref{Ta:3points_50noise}, respectively. The computed source locations remain accurate across all noise levels, with relative errors below \(0.466\%\). While the computed moment vectors are slightly more affected by noise, their relative errors remain within a reasonable range, increasing slightly from \(1.125\%-4.140\%\) for \(10\%\) noise (see Table \ref{Ta:3points}), to \(1.116\% - 4.558\%\) for \(30\%\) noise, and further to \(1.129\% - 4.639\%\) for \(50\%\) noise. These results demonstrate the robustness of the proposed method in identifying sources despite high noise levels in the data.
}
\label{part: noise}
\begin{center}
{\scalebox{0.9}{
    \begin{tabular}{|c|c|c|c|}
    \hline
 \rowcolor{lightgray}    \text{True location} &  \text{Computed location} & \text{True moment} & \text{Computed moment}   \\
        \hline
     $(-0.9,0,1)$ &  $(-0.900, 0.000,1.005)$ &$\begin{pmatrix}
         -2.5\\4\\-3
     \end{pmatrix}$  &  $\begin{pmatrix}
     -2.428\\3.774\\-3.093

     \end{pmatrix}$ 
        \\
     \hline 
     $(-1,0.75,-1)$&$(-1.005, 0.750,-1.005)$& $\begin{pmatrix}
         -1+3i\\5+4i\\3
     \end{pmatrix}$  
     & $\begin{pmatrix}
             -0.923 + 3.006i\\   5.001 + 4.030i\\   2.992i - 0.024i
     \end{pmatrix}$
         \\
         \hline
        $(1.1,-0.3,-1)$&$(1.095, -0.300,-1.005)$ &$\begin{pmatrix}
            4.5i\\-5\\3-2i
        \end{pmatrix}$
    &$ \begin{pmatrix}
         -0.001 + 4.539i  \\-4.951 + 0.016i \\  2.974 - 2.049i
    \end{pmatrix}$
    \\
    \hline
     
    \end{tabular}}
    \captionof{table}{ {Reconstruction results for three point sources  for $k=20, s=4$, and $30\%$ random noise was added to the data}. \label{Ta:3points_30noise}}}
    \end{center}  
\begin{center}
{\scalebox{0.9}{
    \begin{tabular}{|c|c|c|c|}
    \hline
 \rowcolor{lightgray}    \text{True location} &  \text{Computed location} & \text{True moment} & \text{Computed moment}  \\
        \hline
     $(-0.9,0,1)$ &  $(-0.900, 0.000,1.005)$ & $\begin{pmatrix}
         -2.5\\4\\-3
     \end{pmatrix}$  &  $\begin{pmatrix}
     -2.458\\       3.753 \\      -3.067

     \end{pmatrix}$  
        \\
     \hline 
     $(-1,0.75,-1)$&$(-1.005, 0.750,-1.005)$& $\begin{pmatrix}
         -1+3i\\5+4i\\3
     \end{pmatrix}$  
     & $\begin{pmatrix}
              -0.910 + 2.987i \\  4.973 + 4.032i  \\ 3.001 + 0.012i

     \end{pmatrix}$
         \\
         \hline
        $(1.1,-0.3,-1)$&$(1.095, -0.300,-1.005)$ &$\begin{pmatrix}
            4.5i\\-5\\3-2i
        \end{pmatrix}$
    &$ \begin{pmatrix}
           -0.019 +      4.566i    \\   -4.958 +      0.053i     \\    3.03 -      2.011i

    \end{pmatrix}$
    \\
    \hline
     
    \end{tabular}}
    \captionof{table}{ {Reconstruction results for three point sources for $k=20, s=4$, and $50\%$ random noise was added to the data}. \label{Ta:3points_50noise}}}
    \end{center}

\subsection{Comparison between $\widetilde{\widehat{I}_s}(\z)$ and $\widetilde{{I}_s}(\z)$ in imaging sources near the data boundary}
\label{nearfield}
We now evaluate the performance of the imaging function $\widetilde{\widehat{I}_s}(\z)$ defined via $\widehat I(\mathbf{z},\q)$ in \eqref{tildehatI} and compare it with the imaging function $\widetilde{I}_s(\z)$  defined via $I(\mathbf{z},\q)$ in \eqref{tildeI} to reconstruct point sources with real moment vectors. We first consider the case when the Cauchy data are measured on a boundary close to the sampling domain. 
The radius of the measurement sphere $\partial \Omega$ is set to $2.4$, and the sampling cube for our simulations remains $[-1.5,1.5]^3$. Our example involves four point sources, with two located near the measurement boundary and the other two further away. As shown in Table \ref{ihatipts}, $\widetilde{\widehat{I}_s}(\z)$ fails to locate accurately all four sources.
This may be due to the fact that sampling point $\z$ approaches the blow-up singularity of the Green’s tensor in the formulation of $\widehat{I}(\mathbf{z},\q)$ in \eqref{eq:Ihat} when imaging near the data boundary. However, $\widetilde{I}_s(\z)$  accurately determines source locations. 
Now, when the boundary data is not close to the sampling domain, increasing the radius of the measurement sphere to at least $2.6$ in this example, allows $\widetilde{\widehat{I}_s}$ to provide accurate estimates, with results similar to those of $\widetilde{I}_s(\z)$.
\begin{center}
{\scalebox{0.9}{
    \begin{tabular}{|c|c|c|c|}
    \hline
 \rowcolor{lightgray}    \text{True location } ($N=4$) &\text{True moment vector} &  \text{Computed location by   $\widetilde{\widehat{I}_s}(\z)$} & \text{Computed location by $\widetilde{I}_s(\z)$}  \\
        \hline
      $(-1.3,-1.3,-1.3)$& $(-1,-1,-1)$ &$(-1.425, -1.275,-1.275)$
     & $(-1.305, -1.305,-1.305)$
        \\
     \hline 
   
      $(1.4,1.4,1.4)$ & $(1,1,1)$& $(1.470, 1.230,1.275)$
 & $(1.455, 1.365,1.410)$  
         \\
         \hline
          $(-1,-1,0)$&$(-1,-1,-1)$&$(-1.425, -1.275,1.275)$ &$(-1.005, -1.005,0.000)$
       
    \\
     \hline
        $(0.7,0.5,0)$&$(1,1,1)$&$(1.470, -1.230,1.275)$ 
    & $ (0.705, 0.495,0.000)$
    \\
    \hline
    \end{tabular}}
    \captionof{table}{Comparison of the reconstruction results using $\widetilde{\widehat{I}_s}(\z)$ and $\widetilde{I}_s(\z)$ for identifying point sources located near the data boundary  {for $k=20, s=4$}. 
    \label{ihatipts}}} 
    \end{center}   
\section{Conclusion}
\label{se:conclu}
We introduce a new numerical method for solving the inverse source problem associated with Maxwell's equations at a fixed frequency Cauchy data. We justify the imaging functions and establish a computational algorithm to determine point sources with (possibly complex) moment vectors that can have notably different magnitudes. Even in the presence of noise in the data, our method accurately and efficiently localizes electromagnetic point sources and small-volume sources.  It can also estimate the moment vectors of point sources and the directions of these vectors for small-volume sources with a low relative error, demonstrating the robustness of our approach.  
The method allows imaging at arbitrary distances from the data boundary, and it offers easy implementation and low computational costs. This makes it a promising approach for addressing inverse source problems.

\vspace{0.5cm}

\noindent{\bf Acknowledgments:} The research of I. Harris is partially supported by  NSF  Grant DMS-2107891. 
The research of T. Le and D.-L. Nguyen is partially supported by NSF Grant DMS-2208293. 

\bibliographystyle{plain}
\bibliography{Imylib}
\end{document}